\documentclass[11pt]{article}

\usepackage{tikz}
\usepackage{adjustbox}
\usepackage{subcaption}
\usepackage{caption}

\usepackage{hyperref}
\usepackage[utf8]{inputenc}
\usepackage{color}
\usepackage{amsmath,amsfonts,amssymb,graphicx}
\usepackage{amsthm}
\usepackage{mathtools}
\usepackage[top=2.8cm,bottom=2.8cm,left=2.5cm,right=2.5cm]{geometry}
\usepackage{algorithm}
\usepackage{algorithmic}
\usepackage{soul}
\usepackage{bm}
\usepackage[toc]{appendix}
\usepackage{multirow}
\usepackage{enumitem}

\usepackage{authblk}
\usepackage{siunitx}
\usepackage{comment}
\usepackage{todonotes}

\newtheorem{lemma}{Lemma}
\newtheorem{remark}{Remark}
\newtheorem{proposition}{Proposition}
\newtheorem{corollary}{Corollary}
\newtheorem{definition}{Definition}
\newtheorem{theorem}{Theorem}

\newtheorem*{example*}{Example}

\def\sgn{{\rm sgn}}

\title{Fourier Residual Networks Achieve Spectral Accuracy\\for Discontinuous Functions} 

\author[1]{Owen Davis \thanks{\texttt{ondavis@sandia.gov}}}
\affil[1]{{\small{Department of Uncertainty Quantification and Optimization, Sandia National Laboratories, Albuquerque, USA}}}
\author[2]{Mohammad Motamed \thanks{\texttt{motamed@unm.edu}}}
\affil[2]{{\small{Department of Mathematics and Statistics, University of New Mexico, Albuquerque, USA}}}
\author[3]{Olof Runborg \thanks{\texttt{olofr@kth.se} (Corresponding author)}}
\affil[3]{{\small{Department of Mathematics, KTH Royal Institute of Technology, Stockholm, Sweden}}}
\date{\today}

\begin{document}

\maketitle

%\bigskip
\hrule
\bigskip

\begin{abstract}
We present a constructive approximation framework for analyzing the expressive power of Fourier residual networks in approximating a broad class of one-dimensional functions. Our study covers both piecewise continuous functions---including those with jump discontinuities in the function and its derivatives---and fully smooth functions. We show that Fourier residual networks achieve spectral convergence without requiring periodicity or continuity, thereby overcoming key limitations of classical linear Fourier approximation and nonlinear methods, without being restricted to Barron-type function spaces. Our approach builds on classical techniques from approximation theory, including fixed-point iteration and Hermite interpolation by trigonometric polynomials. We support our theoretical results with numerical experiments based on both the constructed approximations and a randomized algorithm developed in our earlier work.
\end{abstract}
\smallskip

\bigskip
\hrule

\section{Introduction}
\label{sec:intro}

Two challenges commonly arise when approximating non-periodic, piecewise-smooth functions with jump discontinuities: i) the lack of spectral convergence, and ii) the persistence of Gibbs-like oscillations near discontinuities. These phenomena are well known in classical spectral methods, but they also appear in deep learning models used to approximate functions with sharp features. We refer to these combined challenges as the \emph{roughness barrier}. In this work, we demonstrate that Fourier residual networks, 
which can be viewed as a form of compositional spectral representation, are capable of overcoming the roughness barrier, thereby achieving high-resolution approximations even in the presence of discontinuities.

To better understand the significance of our approach, we briefly revisit classical spectral approximation methods and modern neural network models, and how they contend with the roughness barrier. In doing so, we first clarify what we mean by \emph{spectral convergence}. In this work, spectral convergence means that for every $p > 0$, there exists a constant $C_p > 0$ such that the approximation error satisfies
\[
\text{error} \le C_p N^{-p},
\]
where $N$ denotes the number of degrees of freedom, e.g., the number of terms in a series or the number of network parameters. This corresponds to convergence faster than any algebraic rate (for which $p$ is fixed). We emphasize that this is weaker than exponential convergence, which would require a bound of the form $\text{error} \le C \rho^{-N}$ for some $\rho > 1$.

In classical spectral approximation, global expansions in orthogonal bases---such as Fourier series---converge spectrally for infinitely smooth periodic functions and exponentially for analytic periodic functions. However, for non-periodic, piecewise-smooth functions, these methods fail to achieve uniform convergence across the domain and are plagued by spurious oscillations near discontinuities---an artifact known as the \emph{Gibbs phenomenon}; see, e.g., \cite{Hewitt1979,GottliebShu1997}. These oscillations remain bounded away from zero as the number of terms increases, and the convergence is at best algebraic, with the rate determined by the global smoothness of the function. 
This limitation has long motivated the development of specialized reconstruction techniques that post-process truncated Fourier data in order to recover higher accuracy. This is often done in the setting of smooth but non-periodic functions, where artificial periodic extension introduces boundary discontinuities, as considered in Section~\ref{sec:nonperiodic}.
One of the simplest approaches is filtering, in which high-frequency modes are damped by smoothly decaying multipliers \cite{GottliebShu1997,Tadmor2007}. While this does not improve the overall $L^2$-error, since the original Fourier coefficients are already optimal in that norm, it can yield spectral convergence in the maximum norm on compact subsets away from discontinuities. 
More sophisticated techniques include spectral reprojection, where the truncated Fourier series is re-expanded in alternative polynomial bases, such as Gegenbauer or Freud polynomials, \cite{GottliebShu1997,GelbTanner2006}. These methods can achieve spectral convergence in the maximum norm up to the location of the discontinuity. Related approaches, such as inverse polynomial reconstruction \cite{JungShizgal2004,Pasquetti2004,HrycakGrochenig2010,AdcockHansen2012}, further stabilize this reprojection process. Another notable class of methods is singular Fourier Pad\'e approximation \cite{DriscollFornberg2001,Beckermann_etal2011}, which incorporates information about the singular structure of the function through quotients of trigonometric polynomials in order to achieve spectral accuracy in a similar sense. Despite these advances, fundamental limitations remain. In particular, for analytic but non-periodic functions it is known that any stable reconstruction procedure based solely on finitely many Fourier coefficients can converge at most root-exponentially, i.e.,
$\sim \rho^{-\sqrt{N}}$,
with respect to the number $N$ of retained modes. Any method achieving faster convergence must necessarily be unstable; see, e.g., \cite{AdcockHansenShadrin2014,Boyd2005trouble}.
A detailed comparison between our compositional Fourier approximation approach and these reconstruction techniques is beyond the scope of the present work and is left for future investigation.

In neural network approximation, similar challenges to those in spectral methods arise when approximating functions with jump discontinuities. Classical results show that neural networks---particularly those with ReLU or other non-polynomial activations---can, in principle, approximate a broad class of functions \cite{cybenko1989approximation,Hornik_etal:89,leshno1993multilayer}. For piecewise-smooth targets, more recent work has established sharp algebraic approximation rates and optimal complexity bounds for deep ReLU networks \cite{petersen2018optimal}. At the same time, periodic activations have been shown to enable very fast approximation rates for globally smooth H\"older classes through highly expressive ``deep Fourier'' constructions \cite{yarotsky2019phase}. However, these results address different regimes from the one considered here: the former concerns standard ReLU architectures with algebraic rates for broad piecewise-smooth classes, while the latter concerns the approximation of globally smooth targets using expressive periodic-activation constructions based on encoding and lookup of function information. In particular, neither directly addresses the classical Fourier roughness barrier for non-periodic, piecewise-smooth functions with jumps that is the focus of the present work.

The approximation results in this paper concern expressivity: we ask whether there exist Fourier residual network parameters that achieve high accuracy with controlled network size, independently of how those parameters are found. This is distinct from, but closely related to, the practical question of whether training algorithms can efficiently find such parameters. In general, a significant gap remains between expressive guarantees and the performance of trained networks \cite{AdcockDexter2021}. One important source of this gap is spectral bias: networks trained via first order gradient-based optimization tend to prioritize learning low-frequency, smooth features over high-frequency, sharp, or discontinuous ones \cite{Rahaman_etal2019,Xu_etal2019,Basri_etal2020}. This optimization-driven effect can impede accurate representation of high-frequency or discontinuous features in practice, even when suitable approximating parameters exist. Furthermore, Gibbs-like oscillations---artifacts similar in form to those observed in spectral approximations---have been reported in various contexts involving sharp transitions, e.g., in physics-informed neural networks, implicit neural representations, and function regression with multilayer perceptrons \cite{Wang_etal2021,Bubeck_etal2021}. Altogether, these observations suggest that deep networks face two related but distinct challenges for piecewise-smooth or discontinuous targets: the expressivity challenge of representing such functions efficiently, and the algorithmic challenge of finding good representations through training.

In contrast to existing approaches, we demonstrate, via constructive approximation, that there exist deep Fourier residual networks that overcome the roughness barrier in a precise expressivity sense. 
Our theoretical results build on classical approximation theory techniques, including fixed-point iteration and Hermite interpolation by trigonometric polynomials. 
Our network architecture employs complex exponential (or trigonometric) activation functions within a residual-style framework and avoids direct dependence on initial Fourier coefficients. Instead, it distributes non-uniform wave numbers (frequencies) across the network's width and leverages a compositional structure through its depth. We establish quantitative approximation bounds for this class of architectures, which apply even to non-periodic functions with localized discontinuities (see Theorems~\ref{thm:step}--\ref{thm:main}). 
In particular, for piecewise-$C^\infty$ functions with jump discontinuities, the resulting approximation exhibits spectral convergence in the width parameter and exponential convergence in depth. In the specific case of step functions, we show that deep Fourier residual networks achieve exponential accuracy in depth in a fully monotonic manner, without undershoots or overshoots. For more general functions, the approximations may still exhibit oscillatory behavior, but the support of these oscillations shrinks rapidly as the number of parameters increases. This provides a mechanism for resolving the Gibbs phenomenon in an asymptotic sense, through increasingly localized oscillatory regions, in a manner that differs fundamentally from classical spectral reconstruction approaches.

Beyond these expressivity results, Fourier residual networks can be trained using the random-sampling-based algorithm introduced in our earlier work \cite{Davis_etal:2025}, which was designed in part to mitigate spectral bias by avoiding gradient-based optimization. Although the theory developed here concerns expressivity rather than optimization, this connection naturally raises the practical question of whether the predicted rates can also be realized computationally. In the numerical experiments, we therefore test both the explicit constructive approximations used in the proofs and the extent to which this existing training algorithm empirically realizes the predicted approximation behavior.

From a theoretical perspective, our work also advances neural network approximation theory by moving beyond traditional smoothness assumptions. Much of the existing theory, including recent developments for Fourier networks, has been formulated within the framework of spectral Barron spaces introduced in \cite{Barron1993} (see also \cite{KlusowskiBarron2018,LiaoMing:2025} and references therein). These spaces characterize target functions through the decay of their Fourier spectra and impose global smoothness and continuity. While well-suited for smooth function approximation, this framework does not accommodate piecewise-smooth or discontinuous targets. This limitation applies in particular to recent works on Fourier networks, including those based on randomized Fourier features and residual architectures \cite{1layerKammonen,kammonen2023smaller,DavisMotamed:2024}. These studies operate within the spectral Barron space framework and therefore restrict their approximation guarantees to globally smooth functions. In contrast, our analysis does not rely on the inverse Fourier transform representation of the target function. This allows us to relax the global smoothness assumptions inherent in previous work and extend the approximation theory of deep Fourier networks to a broader class of functions, including those with jump discontinuities.

The remainder of this paper is organized as follows. In Section \ref{sec:FNs}, we review the architecture of Fourier residual networks and briefly summarize the randomized training algorithm introduced in our earlier work, which we later use in the numerical experiments. Section \ref{sec:step} presents a motivating example by constructing deep Fourier networks that approximate step functions with exponential accuracy. In Section \ref{sec:generalization}, we extend this analysis to more general piecewise-smooth functions and establish spectral convergence under mild regularity assumptions. Section \ref{sec:numerics} provides numerical experiments that support our theoretical findings and illustrate the practical effectiveness of the proposed architecture. Finally, Section \ref{sec:conclusions} offers concluding remarks and outlines directions for future work.

%%%%%%%%%%%%%%%%%%%%%%%%%%%%%%%%%%%%%%%%%%
\section{Fourier Residual Networks}
\label{sec:FNs}

A Fourier residual network (Fourier ResNet) of depth $L \ge 2$ and width $W\ge 1$
is a real-valued function of one variable, 
$f_L: {\mathbb R} \rightarrow {\mathbb R}$, defined recursively as:
\begin{align*}
  f_1(x) &= \underbrace{\Re \sum_{k=1}^{W} c_{1,k} \, e^{i\omega_{1,k} x}}_{g_1(x)},\\
f_{\ell}(x) &= f_{\ell-1}(x) + \underbrace{\Re \sum_{k=1}^{W} c_{\ell,k} \, e^{i\omega_{\ell,k}x}}_{g_{\ell}(x)} + \underbrace{\Re \sum_{k=1}^{W} c'_{\ell,k} \, e^{i\omega'_{\ell,k}f_{\ell-1}(x)}}_{h_{\ell}(f_{\ell-1}(x))}, \qquad \ell=2, \dotsc, L.
\end{align*}
Here, $\omega_{\ell,k}, \omega_{\ell,k}'\in\mathbb{R}$ and $c_{\ell,k},c_{\ell,k}'\in\mathbb{C}$ are respectively frequency and amplitude parameters of the network. Figure~\ref{fig:FNN} illustrates the residual-style architecture of the Fourier network through a graph-based schematic.
\begin{figure}[!ht]
    \centering
    \includegraphics[width=0.99\linewidth]{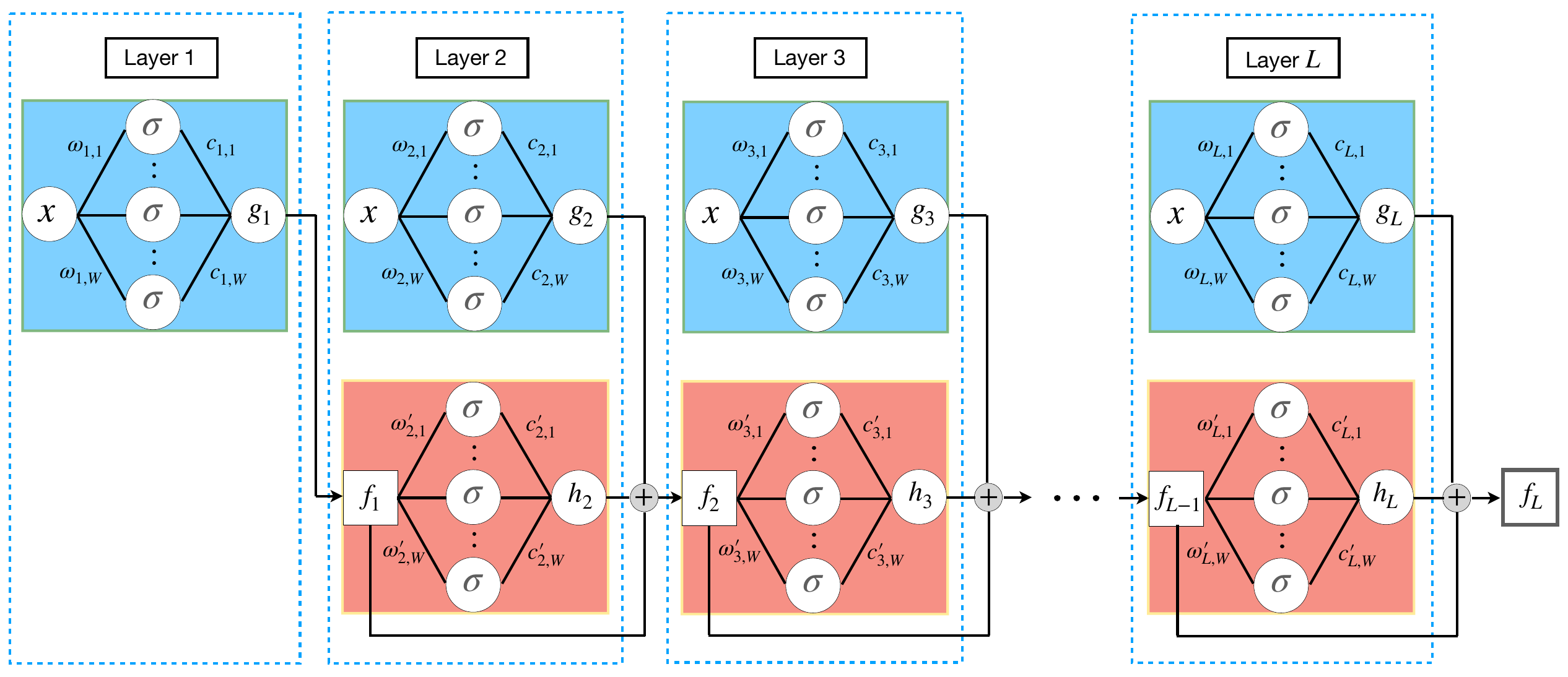}
    \caption{Schematic of the Fourier network \( f_L \), which employs a residual-style architecture. The first layer consists of \( W \) neurons and takes \( x \) as input, producing the output \( f_1(x) = g_1(x) \). Each subsequent layer \( \ell \ge 2 \) has two parallel branches, each with \( W \) neurons: an upper branch that computes \( g_{\ell}(x) \) from \( x \), and a lower branch that computes \( h_{\ell}(f_{\ell-1}(x)) \) from \( f_{\ell-1}(x) \), which is the output of the previous layer. The outputs of both branches are added to the previous layer's output \( f_{\ell-1}(x) \), yielding the updated output \( f_{\ell}(x) \). Here, $\sigma$ denotes the complex exponential activation function.}
    \label{fig:FNN}
\end{figure}

Utilizing Euler's formula $e^{i \, \theta} = \cos \theta + i \, \sin \theta$, an explicit real-variable formulation of Fourier ResNets can be expressed as:
\begin{align}
f_{1}(x) &= g_{1}(x;\bm{\omega}_{1},\bm{a}_{1},\bm{b}_{1}), \label{eq:real1}\\
f_{\ell}(x) &= f_{\ell-1}(x) + g_{\ell}(x;\bm{\omega}_{\ell},\bm{a}_{\ell},\bm{b}_{\ell}) + h_{\ell}(f_{\ell-1}(x);\bm{\omega}_{\ell}',\bm{a}_{\ell}',\bm{b}_{\ell}'), \qquad \ell=2,\dotsc, L, \label{eq:real2}
\end{align}
with real-valued functions $g_{\ell}$ and $h_{\ell}$ explicitly given by: 
\begin{align}
  g_{\ell}(x) &= \sum_{k=1}^{W} a_{\ell,k}\sin(\omega_{\ell,k}x)+b_{\ell,k}\cos(\omega_{\ell,k}x), \qquad \ell = 1,\dotsc, L, \label{eq:g}\\
  h_{\ell}(f_{\ell-1}(x)) &= \sum_{k=1}^{W} a'_{\ell,k}\sin(\omega'_{\ell,k} f_{\ell-1}(x))+b'_{\ell,k}\cos(\omega'_{\ell,k}f_{\ell-1}(x)), \qquad \ell = 2,\dotsc, L. \label{eq:gp}
\end{align}

Fourier networks were first introduced in \cite{1layerKammonen} as single hidden layer networks (see Block 1 in Figure~\ref{fig:FNN}) capable of approximating functions $f\in L^{1}(\mathbb{R}^d)$ that admit a pointwise Fourier representation
\begin{equation}\label{eq:FourierRep}
f(x) = \int_{\mathbb{R}^d} \hat{f}(\omega) e^{i  \omega \cdot x} \, d\omega, \quad \text{for all } x \in \mathbb{R}^d,
\end{equation}
for some complex-valued function $\hat{f}\in L^{1}(\mathbb{R}^d)$. Functions with these properties belong to the spectral Barron space of order zero \cite{KlusowskiBarron2018,LiaoMing:2025}, and under these conditions, they are uniformly continuous and vanish at infinity, reflecting global smoothness and rapid spectral decay; see, e.g., \cite[Corollary 1.21]{SteinWeiss1971}. 

Given that the target function has an inverse Fourier integral representation \eqref{eq:FourierRep}, a one-layer Fourier network can be viewed as a trainable Monte Carlo estimator of this integral. This relationship is established in \cite{1layerKammonen} by assuming the network's frequency parameters are independent and identically distributed random variables following a distribution $p_1:\mathbb{R}^d\rightarrow [0,\infty)$. Then, as in importance sampling, the inverse Fourier integral can be estimated as
\begin{equation}
f(x) = \int_{\mathbb{R}^d} \frac{\hat{f}(\omega) e^{i \omega \cdot x} \,p_1(\omega)}{p_1(\omega)} d\omega\approx \Re \sum_{k=1}^{W}\frac{\hat{f}(\omega_k)e^{i\omega_k\cdot x}}{Wp_1(\omega_k)} = f_{MC}^{W},
\end{equation}
where $f_{MC}^{W}$ is a $W$-term Monte Carlo estimator of $f$. Block 1 of the  network $f_{1}$ and the Monte Carlo estimator $f_{MC}^{W}$ coincide when the network amplitudes are chosen as
$$c_{1,k} = \frac{\hat{f}(\omega_{1,k})}{Wp_1(\omega_{1,k})}, \qquad k=1, \dots, W.$$

This relationship enables the authors in \cite{1layerKammonen} to derive a closed-form upper bound on the network's generalization error, which depends on the frequency distribution $p_1$. This upper bound is minimized over $p_1$ to obtain an optimal distribution of frequencies $p_{1}^{*}$ that satisfies the proportionality $p_{1}^{*}(\omega_1)\propto |\hat{f}(\omega_1)|$. Utilizing this optimal distribution, they develop a training algorithm that employs a Metropolis-Hastings procedure to approximately sample the network's frequency parameters from $p_{1}^{*}$.

Later, in \cite{kammonen2023smaller}, the deep Fourier residual network architecture was introduced and its approximation capabilities were studied for target functions in spectral Barron space of order zero. Notably, this work demonstrated that the generalization error for deep Fourier residual networks can be lower than for one hidden layer networks for target functions satisfying $||\hat{f}||_{L^{1}(\mathbb{R}^d)}/||f||_{L^{\infty}(\mathbb{R}^d)}\gg 1$. This work also explored training deep Fourier residual networks by first sampling the frequencies $\omega_{\ell}$ in each block using the optimal sampling algorithm from \cite{1layerKammonen}, and then conducting global gradient-based optimization over all network parameters simultaneously.

Subsequently, in \cite{Davis_etal:2025}, an optimal sampling algorithm was developed to train deep Fourier residual networks in a block-by-block manner, requiring no subsequent global gradient-based optimization, and enabling an adaptive network architecture that is expanded during training. The process begins with a Block 1, which is trained using the optimal sampling algorithm from \cite{1layerKammonen}. Additional blocks are then added and trained sequentially until a specified tolerance is reached. Crucially, at any Block $\ell \geq 2$, there are two distinct types of frequency parameters $\omega_{\ell}$ and $\omega'_{\ell}$ that are sampled from different a priori derived optimal distributions $p^{*}_{\ell}(\omega_{\ell})$ and $q^*_{\ell}(\omega_{\ell}')$ specific to that block. For a comprehensive discussion of the sampling algorithm and its advantages we refer readers to \cite{Davis_etal:2025}.

In contrast with previous works on Fourier networks \cite{1layerKammonen, kammonen2023smaller}, we consider target functions $f \in {C_{\mathrm{cl}}^m}((-1,1) \setminus \{0\})$ for some $m \geq 0$, which may exhibit a discontinuity at $x = 0$. Notably, such functions lack a classical pointwise inverse Fourier integral representation, and our theoretical results do not depend on this representation. 

However, in Section~\ref{sec:numerics}, we evaluate the performance of the optimal sampling algorithm from \cite{Davis_etal:2025} in learning these discontinuous target functions, acknowledging that the theoretical foundation for the sampling algorithm presupposes an inverse Fourier integral representation. To accommodate this, we adopt a relaxed definition of the inverse Fourier transform. Specifically, if $f \in {C_{\mathrm{cl}}^m}((-1,1) \setminus \{0\})$ is extended outside $(-1,1)$ to a function $\tilde{f} \in C_{\mathrm{c}}^m(\mathbb{R} \setminus \{0\})$ via a smooth, compactly supported continuation that coincides with $f$ on $(-1,1)$, then $\tilde{f} \in L^1(\mathbb{R}) \cap L^2(\mathbb{R})$. However, its Fourier transform $\hat{f}$ may not belong to $L^1(\mathbb{R})$, which prevents the direct application of the classical pointwise inversion formula. Nevertheless, since \( \tilde{f} \in L^1(\mathbb{R}) \), its Fourier transform \( \hat{f} \) is well-defined; moreover, since \( \tilde{f} \in L^2 (\mathbb{R}) \), Plancherel's theorem implies \( \hat{f} \in L^2(\mathbb{R}) \). In this context, the results of Carleson~\cite{Carleson1966} and Hunt~\cite{Hunt1968} guarantee that the Fourier inversion formula holds pointwise almost everywhere:
\[
f(x) = \lim_{R \to \infty} \int_{-R}^{R} \hat{f}(\omega) e^{i  \omega x} \, d\omega, \quad \text{for almost every } x \in (-1,1).
\]
 For detailed discussions see, \cite{Stein1970,SteinWeiss1971,Grafakos2014}. In the present work, this relaxed, pointwise almost everywhere existence of the inverse Fourier integral representation serves as justification for the use of the optimal sampling algorithm in Section~\ref{sec:numerics}.

%%%%%%%%%%%%%%%%%%%%%%%%%%%%%%%%%%%%%
\section{A Motivating Example: Approximating Step Functions}
\label{sec:step}

As an example, this section explores how well Fourier ResNets can approximate piecewise-constant functions, using the sign function as a simple case study. 

Let $x \in [-1,1]$. Consider the sign function:
\begin{equation}\label{eqn:sgn_function}
f(x) = \sgn(x) := \begin{cases}
  1, & 0<x\leq 1,\\
  0, & x=0,\\
  -1, & -1\leq x<0.
  \end{cases}
\end{equation}
It is well-known that the $N$-term Fourier series approximation of this step function, which corresponds to a one-layer Fourier network with width $W = N$ and integer frequencies, suffers from Gibbs' Phenomenon. This results in spurious oscillations near the discontinuity at $x=0$, where the Fourier approximation overshoots or undershoots the step function. Importantly, the magnitude of this error remains ${\mathcal O}(1)$, regardless of how large $N$ becomes (i.e., for any finite $N$).

Instead, we consider the following approximation of the sign function:
\begin{align}
f_{1}(x) &= \sin\left(\frac{\pi}2 \, x\right), \label{eq:f0def}\\
f_{\ell}(x) &= f_{\ell-1}(x) + \frac{1}\pi\sin\left({\pi}f_{\ell-1}(x)\right), \qquad \ell=2,\dotsc, L. 
\label{eq:fndef}
\end{align}
This defines a Fourier ResNet with width $W = 1$, which takes the form of the equations \eqref{eq:real1}-\eqref{eq:gp}, with the following parameters:
$$
  \omega_{1,1} =\frac{\pi}{2}, \quad a_{1,1} = 1, \quad b_{1,1} = 0, 
$$
$$
  \omega_{\ell,1}=a_{\ell,1}=
  b_{\ell,1}= b'_{\ell,1} = 0, \qquad 
  \omega'_{\ell,1}=\pi, \quad
  a'_{\ell,1}=1 / \pi, \quad \ell \ge 2.
$$
Our first result shows that as the depth $\ell$ increases, the sequence ${f_{\ell}}$ converges to the sign function at an exponential rate in $L^p$-norms.
Furthermore, all functions in the sequence $\{f_{\ell}\}_{\ell \ge 1}$ are strictly increasing, with no overshoots or Gibbs-like phenomena. The precise result of our analysis is formulated in the following theorem.

\begin{theorem}\label{thm:step}
For all $\ell \geq 1$, the functions $f_{\ell}$ defined in \eqref{eq:f0def} and \eqref{eq:fndef} satisfy the following properties: 
\begin{align} 
&f_{\ell}\colon [-1,1] \to [-1,1], \label{eq:fnprop1}\\ 
&f_{\ell} \text{ is strictly increasing on } [-1,1],\label{eq:fnprop2}\\ 
&f_{\ell}(-x) = -f_{\ell}(x) \text{ for all } x \in [-1,1].\label{eq:fnodd} 
\end{align} 
Moreover, for $x\in[-1,1]$,
\begin{equation}\label{eq:fnconv}
  \lim_{\ell \to \infty}f_{\ell}(x)
  =\sgn(x),
\end{equation}
and for any $p \in (0, \infty)$, there exists a constant $C_p > 0$, depending only on $p$, such that 
\begin{equation}\label{eq:fnerrestimate} 
\| \sgn - f_{\ell} \|_{L^p(-1,1)} \leq C_p \, 2^{- \ell / p}.
\end{equation}  
\end{theorem}

The proof of Theorem \ref{thm:step} relies on analyzing the fixed-point iteration associated with the recursive definition \eqref{eq:f0def}–\eqref{eq:fndef}. For fixed $x \in [-1,1]$, the sequence $\{ f_{\ell}(x)\}_{\ell \ge 1}$ is defined by the iteration:
$$
 f_{\ell+1}(x) = \phi(f_{\ell}(x)), \qquad \ell \ge 1,
$$
where the iteration function $\phi: {\mathbb R} \rightarrow {\mathbb R}$ is given by:
\begin{equation}\label{eq:fpfunc}
\phi(y) = y+\frac{1}\pi\sin\left({\pi}y\right).
\end{equation}

We begin by summarizing key properties of the iteration function $\phi$ in the following lemma.

\begin{lemma}\label{lem:phiprop}
Let $\phi$ be defined as in \eqref{eq:fpfunc}, with $y \in [-1,1]$. 
Then the following hold:
\begin{align}
 &\text{$\phi$ maps $[-1,1]$ into $[-1,1]$}, \label{eq:phiprop1}\\
 &\text{$\phi$ is strictly increasing on $[-1,1]$}, \label{eq:phiprop2}\\
 &\text{$\phi(-y) = -\phi(y)$ for all $y \in [-1,1]$}. \label{eq:phiodd}
\end{align}
\end{lemma}
\begin{proof}
Suppose $y\in(-1,1)$, then 
$$
\phi'(y) = 1 +\cos(\pi y)>0.
$$
Thus, 
since $\phi$ is also continuous on $[-1,1]$,
it is strictly increasing on the closed interval $[-1,1]$. This establishes \eqref{eq:phiprop2}. 
Moreover, the monotonicity implies that $$
\phi(y)\in[\phi(-1),\phi(1)]=[-1,1], \qquad \text{for all} \  y\in[-1,1],
$$
which proves \eqref{eq:phiprop1}. 
Finally, \eqref{eq:phiodd} follows directly from the definition \eqref{eq:fpfunc}, since both $y$ and $\sin(\pi y)$ are odd functions, making $\phi$ itself odd.
\end{proof}

%%%%%%%%%%%%%%%%%%%%%%%%%%%%%%%%%%%%%%%%%%%%%%%%%%%%%%%

Utilizing these properties of the iteration function, we next study the behavior of $f_{\ell}(x)$ by analyzing the sequence $\{ y_{\ell }\}$ defined by
\begin{equation} \label{eq:fpiter} 
y_{\ell+1} = \phi(y_\ell), \qquad \ell = 1, 2, \ldots, 
\end{equation} 
for different initial values $y_1 \in [-1,1]$.

\begin{lemma}\label{lem:fpconv}
Let $\phi$ be defined as in \eqref{eq:fpfunc}, and let the sequence $\{y_\ell\}$ be given by the iteration \eqref{eq:fpiter}. %Then:
\begin{itemize}
  \item[(i)] If $y_1 \in (0,1)$, then $\{y_\ell\} \subset [0,1]$ is strictly increasing and converges to $1$.
  \item[(ii)] If $y_1 \in (-1,0)$, then $\{y_\ell\} \subset [-1,0]$ is strictly decreasing and converges to $-1$.
  \item[(iii)] If $y_1 \in \{-1, 0, 1\}$, then $y_\ell = y_1$ for all $\ell \ge 1$.
\end{itemize}
\end{lemma}
\begin{proof}
Suppose first that $y_1 \in (0,1)$. Then for all $\ell \ge 1$, if  $y_{\ell} \in (0,1)$, we have
$$
y_{\ell+1} = \phi(y_{\ell}) = y_{\ell} + \frac{1}{\pi}\sin(\pi y_{\ell}) > y_{\ell},
$$
since $\sin(\pi y_{\ell})>0$ for $y_{\ell}\in(0,1)$. 
Hence, the sequence is strictly increasing as long as it remains in $(0,1)$.  
By the monotonicity of $\phi$ from Lemma~\ref{lem:phiprop}, we also have $\phi(y) \in [\phi(0),\phi(1)] = [0,1]$ for all $y \in [0,1]$, so $y_{\ell} \in [0,1]$ for all $\ell$, and the sequence is bounded above by 1. Therefore, by the monotone convergence theorem, $y_{\ell} \to y^*\in (y_{\ell},1]$. Taking the limit in the recurrence yields
$$
y^* = \lim_{\ell \to \infty} y_{\ell+1} = \lim_{\ell \to \infty} \phi(y_{\ell}) = \phi(y^*),
$$
so $y^*$ is a fixed point of $\phi$. But the only fixed points of $\phi(y) = y + \frac{1}{\pi}\sin(\pi y)$ are the integers, as $\phi(y) = y$ if and only if $\sin(\pi y) = 0$. Since $y^*\in (y_{\ell},1]$, the only possibility is $y^*=1$. 

Next, suppose $y_1 \in (-1,0)$. Define $\tilde{y}_1 = - y_1 \in (0,1)$, and let $\{\tilde{y}_\ell\}$ denote the sequence generated by the iteration $\tilde{y}_{\ell+1} = \phi(\tilde{y}_\ell)$ for $\ell \geq 1$. Now consider the sequence $\{y_\ell\}$ defined by the same recurrence $y_{\ell+1} = \phi(y_\ell)$, but with initial value $y_1$. 
By the oddness of $\phi$, it follows by induction that $y_\ell = -\tilde{y}_\ell$ for all $\ell \geq 1$. From the previous case, we know that $\tilde{y}_\ell$ is strictly increasing and converges to $1$ as $\ell \to \infty$. Therefore, $y_\ell = -\tilde{y}_\ell$ is strictly decreasing and converges to $-1$. 

Finally, if $y_1 \in \{-1, 0, 1\}$, then $\phi(y_1) = y_1$ because $\sin(\pi y_1) = 0$. Thus the sequence is constant; $y_{\ell}=y_1$ for all $\ell \geq 1$.
\end{proof}

With Lemma \ref{lem:phiprop} establishing structural properties of $\phi$, and Lemma \ref{lem:fpconv} establishing the convergence behavior of the sequence $\{ y_\ell \}$, we now derive explicit estimates on its rate of convergence.

\begin{lemma}\label{lem:ynerrestimate}
Let the sequence $\{y_\ell\}_{\ell \ge 1}$ be generated by the fixed-point iteration \eqref{eq:fpiter}, with the iteration function $\phi$ as in \eqref{eq:fpfunc} and with initial value $y_1 \in [0,1]$. Then the following estimates hold for all $\ell \ge 1$:
\begin{align}
   |1 - y_\ell| &\leq (1 - y_1)^{2^{\ell-1}}, \label{eq:ynerrest1} 
\end{align}
\end{lemma}
\begin{proof}
Let $\varepsilon_\ell :=1-y_\ell$.
Then
$$
\varepsilon_{\ell+1} = 1-\phi(1-\varepsilon_\ell) = 
\varepsilon_\ell - \frac{1}\pi \sin\left({\pi}\varepsilon_\ell \right)=:  \tilde\phi(\varepsilon_\ell).
$$
Taylor expansion of $\tilde{\phi}$ around $0$ gives:
$$
\tilde{\phi}(\varepsilon)=\tilde{\phi}(0)
+ \varepsilon \tilde{\phi}'(0)+\frac{\varepsilon^2}2\tilde{\phi}''(0)
+\frac{\varepsilon^3}6\tilde{\phi}'''(\xi)
=\varepsilon^3\frac{\pi^2}6\cos(\xi), \qquad \text{for some } \xi \in (0, |\varepsilon|\pi).
$$
Thus, for all $\varepsilon$,
\begin{equation}\label{eq:phitildeest}
|\tilde{\phi}(\varepsilon)| \leq \frac{\pi^2}{6} |\varepsilon|^3.
\end{equation}
To prove \eqref{eq:ynerrest1}, we first show that $0\leq \tilde{\phi}(\varepsilon)\leq \varepsilon^2$ for $\varepsilon\in[0,1]$. 
The lower bound comes from 
the inequality $\phi(\varepsilon)\leq 1$
from \eqref{eq:phiprop1} in Lemma~\ref{lem:phiprop}.
The upper bound follows from
\eqref{eq:phitildeest} and
 the fact that $\tilde{\phi}(\varepsilon)
=\tilde{\phi}(1-\varepsilon)-1+2\varepsilon$,
$$
\tilde{\phi}(\varepsilon)\leq 
\begin{cases}
 \frac{\pi^2}{6} \frac{\varepsilon^2}{2}, & \varepsilon\in[0,1/2], \\
 \frac{\pi^2}{6} \frac{(1-\varepsilon)^2}{2}-1+2\varepsilon, & \varepsilon\in[1/2,1],
\end{cases}
\leq
\begin{cases}
 \varepsilon^2, & \varepsilon\in[0,1/2], \\
 (1-\varepsilon)^2
-1+2\varepsilon, & \varepsilon\in[1/2,1],
\end{cases}
=\varepsilon^2.
$$
We use induction to prove
\eqref{eq:ynerrest1}. It holds trivially for $\ell = 1$. Assume it holds for $k = 1, \ldots, \ell$. Since $y_1 \in [0,1]$, Lemma~\ref{lem:fpconv} guarantees that $\varepsilon_k \in [0,1]$ for all $k$. Then,
\[
\varepsilon_{\ell+1} = \tilde{\phi}(\varepsilon_\ell) \le \varepsilon_\ell^2 \le \left(\varepsilon_1^{2^{\ell-1}}\right)^2 
= \varepsilon_1^{2^\ell},
\]
which completes the induction and proves \eqref{eq:ynerrest1}. 
\end{proof}

We are now ready to prove Theorem \ref{thm:step}.

\medskip
\noindent
{\bf Proof of Theorem \ref{thm:step}.} 
We proceed by induction to show that \eqref{eq:fnprop1}--\eqref{eq:fnodd} hold. It is straightforward to verify that these properties are satisfied by $f_1$. 
Assuming they hold for $f_k$ for all $k = 1, \ldots, \ell$, we obtain
\[
f_{\ell+1}(x) = \phi(f_\ell(x)) \in [-1,1],
\]
by \eqref{eq:phiprop1}. Moreover, since $f_\ell$ is strictly increasing and $\phi$ is strictly increasing by \eqref{eq:phiprop2}, their composition $f_{\ell+1}$ is also strictly increasing on $[-1,1]$. Finally, since $f_\ell$ is odd by assumption and $\phi$ is odd by \eqref{eq:phiodd}, we have
\[
f_{\ell+1}(-x) = \phi(f_\ell(-x)) = \phi(-f_\ell(x)) = -\phi(f_\ell(x)) = -f_{\ell+1}(x),
\]
so $f_{\ell+1}$ is odd. 
By induction, the properties \eqref{eq:fnprop1}-\eqref{eq:fnodd} hold for all $\ell$.

The convergence \eqref{eq:fnconv} when $x\in(-1,1)$ follows directly from Lemma~\ref{lem:fpconv} and the fact that $f_1(x)\in(-1,1)$ when $x\in(-1,1)$. 
At the integer points $x \in \{-1, 0, 1\}$, the result is immediate since these are fixed points of $\phi$.

To establish the norm error estimate \eqref{eq:fnerrestimate}, we 
use the oddness property \eqref{eq:fnodd}, 
the inequality
\begin{equation}\label{eq:initdataest}
|1-f_1(x)|\leq |1-x|,\qquad 0\leq x\leq 1,
\end{equation}
and the estimate \eqref{eq:ynerrest1} to obtain
\begin{align*}
\int_{-1}^1 \left| \sgn(x) - f_\ell(x) \right|^p \, dx 
&= 2 \int_0^1 \left| 1 - f_\ell(x) \right|^p \, dx 
\leq 2 \int_0^1 (1 - f_1(x))^{2^{\ell - 1}p} \, dx \\
&\leq 2 \int_0^1 (1 - x)^{2^{\ell - 1}p} \, dx 
= 2 \int_0^1 x^{2^{\ell - 1}p} \, dx 
= \frac{2}{2^{\ell - 1}p + 1}
\leq \frac{4}{p} 2^{-\ell}.
\end{align*}
This proves \eqref{eq:fnerrestimate} with $C_p = (4/p)^{1/p}$.

\begin{comment}
{\color{purple}
MM: Things to consider adding here or moving to Section \ref{sec:numerics}:

\begin{enumerate}
    \item A numerical comparison with Fourier series approximation and a convergence plot
    \item If we decide to include this numerical experiment here, then maybe we should split this section into two: theory and numerics?
    \item Alternatively, we may put this numerical study in Section \ref{sec:numerics}.
\end{enumerate}
}
\end{comment}

%%%%%%%%%%%%%%%%%%%%%%%%%%%%%%%%%%%%%
\section{Generalization: Approximating Piecewise-Smooth Functions}
\label{sec:generalization}

The theory developed in Section~\ref{sec:step} demonstrates that a deep Fourier network can approximate a simple non-periodic step function on the interval $[-1,1]$ with exponential accuracy, without exhibiting Gibbs oscillations. This surprising result motivates us to investigate whether similar approximation guarantees hold for broader classes of functions, particularly those that are not smooth or periodic.

Classical Fourier analysis tells us that high-order approximation typically requires both smoothness and periodicity. However, the step function violates both assumptions, yet we demonstrated that a Fourier network, despite being built from periodic activation functions, can still yield exponential convergence. Note that, while the constructed Fourier network for the step function has the same periodicity as the initial data $f_1(x)=\sin(\pi x/2)$, namely $4$, it is not periodic on our interval of interest $[-1,1]$.

In this section, we extend our analysis to a broader class of functions. We first remove the periodicity assumption in Section~\ref{sec:nonperiodic}, focusing on smooth but non-periodic functions on $[-1,1]$. Then, in Section~\ref{sec:nonsmooth}, we further relax the smoothness assumption by considering piecewise-smooth functions with a single point of discontinuity.

\paragraph{Notation.}
Throughout this section, we use lowercase letters to denote general functions and uppercase letters to denote Fourier network approximations. When needed, subscripts indicate architectural parameters such as depth and width. The parameter $L$ denotes the depth of a Fourier network, while $W$ denotes the width of hidden layers in the sense of Section~2. In addition, we use $N$ to denote the number of Fourier modes in a representation, which corresponds to the width of a shallow Fourier network and determines the approximation resolution (see Theorem~\ref{thm:nonuniformFourier}). In our constructions, $W$ and $N$ are related but play distinct roles: $W$ controls the architectural width of the network, while $N$ reflects the number of effective Fourier terms used in the approximation.

\subsection{Non-periodic smooth functions}
\label{sec:nonperiodic}

We begin by removing the periodicity assumption while retaining smoothness. Specifically, we consider functions that are smooth but not necessarily periodic on the interval $[-1,1]$.

With standard truncated Fourier sums (using integer wave numbers), spectral convergence breaks down once periodicity is lost---even for analytic functions. In such cases, the series converges only linearly on compact subsets of $(-1,1)$. The loss of uniform convergence on the full interval $[-1,1]$ is accompanied by the well-known Gibbs phenomenon near the endpoints $x = \pm 1$; see, e.g., \cite{GottliebShu1997,AdcockHansenShadrin2014}.

Our main result, stated in Theorem~\ref{thm:nonuniformFourier}, shows that high-order approximation can be achieved for sufficiently smooth functions defined on a finite interval but not necessarily periodic, provided that the Fourier expansion includes non-integer wave numbers. This construction corresponds to a single-layer Fourier network and does not require network depth. 
We note that non-periodic functions can also be approximated using filtering techniques, provided that a suitable smooth extension is available; see, e.g., \cite{GottliebShu1997}. However, such approaches require additional assumptions and constructions, whereas the present framework operates directly on the given function.

We will use the following notation to characterize smooth, bounded, and periodic functions.
{
\begin{definition}\label{def:space_smooth}
Let \(m \ge 0\) be an integer. We say that a function
$f\in C^m(a,b)$ belongs to \(C_{\mathrm{cl}}^m(a,b)\) if
all derivatives \(f^{(k)}\), \(0 \le k \le m\) have well-defined finite one-sided limits at \(x=a\) and \(x=b\).
Moreover, we denote by $C^m_{\rm per}(a,b)$ the functions
$f\in C_{\mathrm{cl}}^m(a,b)$ for which
$$
\lim_{x\to a^+} f^{(k)}(x)
  =\lim_{x\to b^-} f^{(k)}(x),\qquad k=0,1,\ldots, m.
$$
\end{definition}
}

\begin{theorem}\label{thm:nonuniformFourier} Let {$m \ge 1$} be an integer, and suppose $f\in {C_{\mathrm{cl}}^m}(-1,1)$, as in Definition~\ref{def:space_smooth}. There exist complex coefficients $\{ \hat{f}_k \} \subset {\mathbb C}$ and real frequencies $\{ \omega_k \} \subset {\mathbb R}$ such that the approximation
$$
F_W(x) = \sum_{k=1}^{N} \hat{f}_k e^{i\omega_k x}, \qquad N = W + 2(m+1),
$$
satisfies the error bound
\begin{equation}\label{eq:estimate_non_periodic}
\| f - F_W  \|_{L^2(-1,1)}\leq C_m W^{-m+1/2},
\end{equation}
for some constant $C_m>0$. Moreover,
$|\hat{f}_k|\leq D_m$ for all $k$.
Both $C_m$ and $D_m$ depend on $f$ and $m$, but
not on $W$.
The frequencies $\{\omega_k\}$
are independent of $f$.
\end{theorem}

This result quantifies the approximation behavior in terms of the available smoothness of the target function. The algebraic convergence rate in \eqref{eq:estimate_non_periodic} reflects the number of bounded derivatives. In particular, as the smoothness increases, the convergence rate improves accordingly, and for sufficiently smooth functions this yields spectral convergence in the number of modes.

\begin{remark}
The construction in Theorem~\ref{thm:nonuniformFourier} uses a total of $N = W + 2(m+1)$ Fourier modes, where the additive offset depends on the smoothness parameter $m$. Since this offset is independent of $W$, it does not affect the asymptotic approximation rate for fixed $m$ and can be absorbed into the width parameter for sufficiently large $N$, cf.\mbox{}
Corollary~\ref{cor:spectral}.
\end{remark}

To prove Theorem \ref{thm:nonuniformFourier}, we rely on two classical results. 
First, if a function has $m$ continuous and periodic derivatives,
then its truncated Fourier series with $W$ terms approximates it with
$L^2$-error of order $\mathcal{O}(W^{-m})$, as summarized below.

\begin{proposition}\label{prop:Fourier}
Let $m\geq 1$ and suppose
$g \in C_{\text{per}}^m(-1,1)$, as in Definition~\ref{def:space_smooth}.
Let
$$
G_W(x) = \sum_{|k| \leq W/2} \hat{g}_k \, e^{i k \pi x}, \qquad \hat{g}_k = \frac{1}{2} \int_{-1}^1 g(x) \, e^{-i k \pi x} \, dx,
$$
be the truncated $W$-term Fourier approximation of $g$. Then there exists a constant $C_m > 0$, which depends on $g$ and $m$ but is independent of $W$, such that
$$
\| g - G_W \|_{L^2(-1,1)} \leq C_m \, W^{-m+1/2}.
$$
\end{proposition}
\begin{proof} 
With $g \in C_{\text{per}}^m(-1,1)$ we can integrate
the formula for $\hat{g}_k$ by parts $m$ times.
All boundary terms vanish due to periodicity, giving the estimate
$$
  |\hat{g}_k|\leq d_m|k|^{-m}, \qquad d_m=\pi^{-m}||g^{(m)}||_{L^\infty(-1,1)}.
$$
For details, see Section 2.9, in particular Theorem 4, of \cite{Boyd2000}. 
Parseval's identity then gives
\[
\| g - G_W \|_{L^2(-1,1)}^2 = \sum_{|k| > W/2} |\hat{g}_k|^2 \le 
d_m^2 \sum_{k > W/2} k^{-2m},
\]
and the tail sum satisfies 
$\sum_{k > W/2} k^{-2m} \leq  d'_m W^{-2m +1}$. Taking square roots yields the desired estimate with $C_m = d_m\sqrt{d'_m}$.
\end{proof}

Second, we use the following result on trigonometric Hermite interpolation, taken from \cite{Delvos1993}.

\begin{proposition}[Proposition 4.2 in \cite{Delvos1993}]\label{prop:Hermite}
Let $0 \leq x_1 < x_2 < \cdots < x_n < \pi$ be $n$ distinct real numbers, and assume that $n$ is even. For any set of complex numbers $\{y_{j,s}\}$ with $1 \leq j \leq n$ and $0 \leq s \leq m$, there exists a unique $\pi$-antiperiodic trigonometric polynomial $S(x)$ of order $N-1$ (i.e., consisting of $N$ terms),
$$
S(x) = \sum_{k=-N/2}^{N/2-1} a_{2k+1} e^{i(2k+1)x}, \qquad N = n(m+1),
$$
with complex coefficients $a_{2k+1} \in \mathbb{C}$, that satisfies the Hermite interpolation conditions
$$
S^{(s)}(x_j) = y_{j,s}, \qquad \text{for } \ 1 \leq j \leq n, \ \ 0 \leq s \leq m.
$$
\end{proposition}

For our purposes, it is convenient to restate this result in a form adapted to endpoint interpolation over the symmetric interval $[-1,1]$.

\begin{proposition}\label{prop:Hermite2}
Let $\{(\alpha_s, \beta_s)\}_{s=0}^m$ be any collection of complex numbers. Then there exists a trigonometric polynomial $H(x)$ on the interval $[-1,1]$, consisting of $N = 2(m+1)$ terms,
\begin{equation}\label{eq:Hform}
H(x) = \sum_{k=-N/2}^{N/2-1} c_{2k+1} \, e^{i\frac{2k+1}{4} \pi x}, \qquad c_{2k+1} \in \mathbb{C},
\end{equation}
that satisfies the Hermite interpolation conditions at the boundary points $x_1 = -1$ and $x_2 = 1$:
$$
H^{(s)}(-1) = \alpha_s, \qquad H^{(s)}(1) = \beta_s, \qquad \text{for } 0 \leq s \leq m.
$$
\end{proposition}
\begin{proof}
This follows almost directly from Proposition~\ref{prop:Hermite} by choosing $n = 2$ interpolation points, namely $x_1 = 0$ and $x_2 = \pi/2$, and constructing a trigonometric polynomial
$$
S(x) = \sum_{k=-N/2}^{N/2-1} a_{2k+1} e^{i(2k+1)x}, \qquad N = 2(m+1),
$$
on the interval $[0, \pi/2]$ that satisfies the Hermite conditions,
$$
  S^{(s)}(0)=\left(\frac{4}{\pi}\right)^s\alpha_s, \qquad
  S^{(s)}(\pi/2)=\left(\frac{4}{\pi}\right)^s\beta_s, \qquad s=0,\ldots, m.
$$
The desired trigonometric polynomial is then $H(x)=S((x+1)\pi/4)$,
which is of the form \eqref{eq:Hform}
with $c_{2k+1}=a_{2k+1}\exp(i(2k+1)\pi/4)$.
\end{proof}

With these results in hand, we are now ready to prove Theorem~\ref{thm:nonuniformFourier}.

\medskip
\noindent
{\bf Proof of Theorem \ref{thm:nonuniformFourier}.} 
Let $H$ be the trigonometric polynomial constructed in Proposition~\ref{prop:Hermite2}, which interpolates the endpoint derivatives of $f$ up to order $m$, {which exist since $f\in C_{\mathrm{cl}}^m(-1,1)$,}
\[
H^{(s)}(-1) = \alpha_s := \lim_{x \to -1^+} f^{(s)}(x), \qquad
H^{(s)}(1) = \beta_s := \lim_{x \to 1^-} f^{(s)}(x), \qquad 0 \le s \le m.
\]
We note that $H$
only depends on $f$ and $m$. Its coefficients can therefore be bounded as
$|c_{2k+1}|\leq d_m$, where $d_m$
depends on $f$ and $m$.

Define the difference $g := f - H$. 
Since $H$ matches the endpoint derivatives of $f$ up to order $m$, it follows that for each $0 \le s \le m$,
$$
\lim_{x \to -1^+} g^{(s)}(x) = \lim_{x \to -1^+} f^{(s)}(x) - H^{(s)}(-1) = 0,
\qquad
\lim_{x \to 1^-} g^{(s)}(x) = \lim_{x \to 1^-} f^{(s)}(x) - H^{(s)}(1) = 0.
$$
Hence $g \in C_{\text{per}}^m(-1,1)$. Let $G_W$ be the truncated $W$-term Fourier series of $g$ using the standard Fourier basis:
$$
G_W(x) = \sum_{|k| \le W/2} \hat{g}_k e^{ik\pi x}.
$$
By Proposition~\ref{prop:Fourier}, we have the spectral approximation bound:
$$
\| g - G_W \|_{L^2(-1,1)} \le C_m W^{-m+1/2},
$$
where $C_m$ depends on $m$ and $f$, since $g$ depends on $f$ and $m$.
Moreover, by the expression for $\hat{g}_k$
in the proposition, it follows that,
for all $k$,
$|\hat{g}_k|\leq ||g||_{L^\infty(-1,1)}=:d'_m$,
which also depend only on $f$ and $m$.

We now define the final approximant as
$$
F_W(x) := H(x) + G_W(x).
$$
This can be written as
$$
F_W(x) = \sum_{k=1}^{W + 2(m+1)} \hat{f}_k \, e^{i\omega_k x},
$$
for some 
frequencies $\{\omega_k\} \subset \mathbb{R}$
and coefficients $\{\hat{f}_k\}\subset{\mathbb C}$. 
Here $\{\omega_k\}$
is the union of the $W$ frequencies in $G_W$,
i.e.\mbox{} $\{k\pi\}_{|k|\leq W/2}$,
and the $2(m+1)$ frequencies in $H$, i.e.\mbox{} 
$\{(2k+1)\pi/4\}_{-m-1\leq k\leq m}$
by \eqref{eq:Hform}. 
Note that $\{\omega_k\}$ contains
non-integer multiples of $\pi$
and that the set of
frequencies is
independent of $f$.

With this construction,
$F_{W}$ thus defines a single-layer Fourier network with $W + 2(m+1)$ neurons and achieves the desired approximation rate:
$$
\| f - F_W \|_{L^2(-1,1)} = \| (H+g) - (H+G_W) \|_{L^2(-1,1)} = \| g - G_W \|_{L^2(-1,1)} \leq C_m \, W^{-m+1/2},
$$
as claimed. The coefficients
are bounded by $D_m:=\max(d_m,d'_m)$,
independent of $W$.
\qed

\subsection{Piecewise-smooth functions with jump discontinuities}
\label{sec:nonsmooth}

We now further extend the analysis to a broader class of non-periodic functions that are piecewise smooth, allowing for discontinuities in the function and its derivatives.

To formalize our setting, we consider functions that are smooth on either side of a single point of discontinuity at \( x = 0 \), where the function and its derivatives up to order \( m \ge 0 \) may exhibit finite jumps. Specifically, we focus on functions that have $m$ continuous
and bounded derivatives on both sides of the origin, but whose one-sided derivatives at \( x = 0 \) may differ. These functions are naturally defined on the open interval \( (-1,1) \), which excludes the endpoints and allows us to concentrate on the local behavior near the singularity. 
To precisely characterize the functions under consideration, we introduce the following space:
\begin{definition}\label{def:space}
Let \( m \ge 0 \) be an integer. The space \( {C_{\mathrm{cl}}^m}((-1,1) \setminus \{0\}) \) consists of all functions \( f \colon (-1,1) \to \mathbb{R} \) such that $f \in {C_{\mathrm{cl}}^m}((-1,0)) \cap {C_{\mathrm{cl}}^m}((0,1))$.
\end{definition}

Our goal is to show that deep Fourier networks can still achieve high-order approximation in this setting by combining two ingredients: (i) a deep Fourier network that captures the discontinuity, and (ii) a shallow Fourier network that approximates the smooth residual. The result below formulates the approximation accuracy that can be achieved using such a construction.

\begin{theorem}\label{thm:main} 
Let {\( m \ge 1 \)} be an integer, and let \( f \in {C_{\mathrm{cl}}^m}((-1,1) \setminus \{ 0 \}) \), as in Definition~\ref{def:space}. 
Then for any $L, W \in \mathbb{N}$, 
there exists a deep Fourier network $F_{\rm FN}$ of depth $L+1$, whose hidden layers have width at most $2$, and whose final layer has width $W + 4(m+1)$, such that 
\begin{equation}\label{eq:main_estimate}
\| f - F_{\rm FN} \|_{L^2(-1,1)} \le C_m \left( 2^{-L/2} + W^{-m+1/2} \right),
\end{equation}
where $C_m > 0$ depends on $f$ and $m$, but not on $W$ or $L$.
Moreover, the network coefficients satisfy
\[
|c_{\ell,k}|, |c'_{\ell,k}| \le D_m,
\]
where $D_m$ depends on $f$ and $m$, but not on $W$ or $L$.
The frequencies $\{\omega_{\ell,k}\}$ and $\{\omega'_{\ell,k}\}$ are independent of $f$.
\end{theorem}

This result shows that deep Fourier networks can resolve discontinuities while maintaining high-order approximation in the smooth regions. The error bound \eqref{eq:main_estimate} reflects a clear separation of roles between depth and width: the depth $L$ controls the resolution of the discontinuity, yielding exponential decay, while the width $W$ governs the approximation of the smooth components, yielding algebraic decay determined by the available smoothness. The construction proceeds by decomposing the target function into a singular component, which captures the jump, and a smooth residual, which is approximated using a shallow Fourier network. This decomposition allows the network to localize oscillatory effects and achieve accurate approximation across the domain despite the presence of discontinuities.

We next record two remarks concerning stability and the role of frequencies in the construction.

\begin{remark}
The uniform bound on the coefficients ensures that the approximation remains stable as accuracy is improved, in the sense that no coefficients diverge as $W$ or $L$ increase. In particular, the present construction does not rely on representations that approximate discontinuities by introducing increasingly steep transitions with unbounded weights. A typical example is the representation of a step function using scaled differences of ReLU functions of the form
\[
\frac{1}{\varepsilon}\max(0,x)-\frac{1}{\varepsilon}\max(0,x-\varepsilon),
\]
where the transition layer has width $\varepsilon$, and the coefficients grow like $\varepsilon^{-1}$ as $\varepsilon \to 0$.
\end{remark}

\begin{remark}
The frequencies used in the construction are independent of the target function $f$ in the single-jump setting considered here. For functions with multiple discontinuities, the frequencies may depend on the number and configuration of the jumps. It is plausible that alternative constructions could be developed in which the frequencies remain independent of the locations of the discontinuities. A detailed investigation of such constructions is left for future work.
\end{remark}

We now state two corollaries that further clarify the approximation properties of these networks.

\begin{corollary}[Spectral convergence for piecewise-$C^\infty$ functions]
\label{cor:spectral}
Suppose $f$ is piecewise-$C^\infty$ with a jump discontinuity at $x=0$. Then for any $p>0$, there exists a constant $C_p>0$ such that, for suitable deep Fourier networks $F_{N,L}$ of depth $L+1$, whose hidden layers have width at most $2$ and whose final layer has width at most $2N$,
\[
\|f-F_{N,L}\|_{L^2(-1,1)} \le C_p\bigl(2^{-L/2}+N^{-p}\bigr),
\qquad \text{when $N\geq 2p+5$.}
\]
In particular, the approximation exhibits spectral convergence in $N$ and exponential convergence in $L$.
\end{corollary}

\begin{proof}[Proof of Corollary~\ref{cor:spectral}]
This follows directly from Theorem~\ref{thm:main}. For any prescribed $p>0$, let $m$ be the integer 
{$m=\lceil p+1/2 \rceil\geq 1$}.
Since $f$ is piecewise-$C^\infty$, the theorem applies for this $m$. Let $N = W + 2(m+1)$. Then Theorem~\ref{thm:main} yields a deep Fourier network $F_{N,L}$ of depth $L+1$, whose hidden layers have width at most $2$, and whose final layer has width $W+4(m+1)$, such that
\[
\|f-F_{N,L}\|_{L^2(-1,1)}
\le C'_m\bigl(2^{-L/2}+W^{-m+1/2}\bigr).
\]
Moreover, $p+1/2\leq m <p+3/2$ and
$N\geq 2(p+3/2+1)>  2(m+1)$, so
\begin{align*}
W^{-m+1/2}
&\leq W^{-p}=(N-2(m+1))^{-p}
=\left(1-\frac{2(m+1)}N\right)^{-p}N^{-p}
\leq 
\left(1-\frac{2(m+1)}{2(m+1)+1}\right)^{-p}
N^{-p}
\\
&=
(2m+3)^{p}N^{-p}<(2p+6)^pN^{-p}=:C_p'N^{-p},
\end{align*}
and the final width satisfies
$W+4(m+1)=N+2(m+1)< 2N$.
Then,
\[
\|f-F_{N,L}\|_{L^2(-1,1)}
\le C_p\bigl(2^{-L/2}+N^{-p}\bigr),
\]
with $C_p=C_m'C_p'$. 
\end{proof}

\begin{corollary}[Approximation complexity for piecewise-$C^m$ functions]\label{cor:piecewise_spectral} 
Let $\varepsilon > 0$ and suppose $f$ satisfies the assumptions of Theorem~\ref{thm:main} for some 
fixed $m \ge 1$. Then there exists a Fourier network $F_{\rm FN}$ with depth $L = \mathcal{O}(\log \varepsilon^{-1})$ and width $W = \mathcal{O}(\varepsilon^{-1/(m-1/2)})$ such that
\[
\| f - F_{\rm FN} \|_{L^2(-1,1)} \le \varepsilon.
\]
Moreover,
the network coefficients are bounded
independently of $\varepsilon$.
\end{corollary}

\begin{proof}[Proof of Corollary~\ref{cor:piecewise_spectral}]
Given any $\varepsilon > 0$, we aim to choose $L$ and $W$ such that the right-hand side of the estimate in \eqref{eq:main_estimate} is bounded by $\varepsilon$, i.e.,
\[
C_m \left( 2^{-L/2} + W^{-m+1/2} \right) \le \varepsilon.
\]
It suffices to require
\[
2^{-L/2} \le \frac{\varepsilon}{2C_m} 
\quad \text{and} \quad 
W^{-m+1/2} \le \frac{\varepsilon}{2C_m}.
\]
Solving these inequalities yields
\[
L \ge 2 \log_2\!\left( \frac{2C_m}{\varepsilon} \right) = \mathcal{O}(\log \varepsilon^{-1}),
\quad
W \ge \left( \frac{2C_m}{\varepsilon} \right)^{1/(m-1/2)} = \mathcal{O}(\varepsilon^{-1/(m-1/2)}).
\]
Finally, the network coefficients
$c_{\ell,k}$ and
$c'_{\ell,k}$ are bounded by $D_m$
which does not depend on $\varepsilon$.
This completes the proof.
\end{proof}

These two corollaries characterize the approximation behavior of deep Fourier networks across different regularity regimes. For piecewise-$C^\infty$ functions, the approximation achieves spectral convergence in $N$ and exponential convergence in depth. In contrast, for functions with finite smoothness $m$, the approximation exhibits algebraic decay in the width parameter and logarithmic dependence of the depth on the target accuracy $\varepsilon$. In particular, Corollary~\ref{cor:piecewise_spectral} shows that high-order approximation is retained even in the presence of a discontinuity, with logarithmic depth and algebraic dependence of the width on $\varepsilon$. We note that for the case $m = 0$, where the function $f$ may be nowhere differentiable, the approximation error in our construction remains of order $\mathcal{O}(1)$ regardless of the network depth. Therefore, the corollary, which requires $m \ge 1$, excludes this case.

We now turn to the proof of Theorem~\ref{thm:main}. The argument is based on constructing a function in ${C_{\mathrm{cl}}^m}((-1,1) \setminus \{ 0 \})$ that realizes prescribed jumps in its derivatives up to order $m$ at the point of discontinuity $x=0$. This construction is formalized in the following lemma.

\begin{lemma}\label{lem:axuiliary}
Let $z \in {C_{\mathrm{cl}}^m}((-1,1) \setminus \{0\})$, and suppose that
$$
z(0^-) = -1, \qquad z(0^+) = 1, \qquad z'(0^-) \ne 0, \qquad z'(0^+) \ne 0.
$$
Then, for any prescribed values $\{ \alpha_s, \beta_s \}_{s=0}^m$, there exists a trigonometric polynomial $H(y)$ of degree at most $2(m+1)$ such that the composite function
$$
q(x) := z(x) + H(z(x))
$$
satisfies
\begin{equation}\label{eq:jump_Q}
q^{(s)}(0^-) = \alpha_s, \qquad q^{(s)}(0^+) = \beta_s, \qquad 0 \le s \le m,
\end{equation}
\end{lemma}
\begin{proof}
Let $w(x) := H(z(x))$, where $H$ is the
trigonometric polynomial to be determined. For $x \ne 0$, the chain rule and induction give:
$$
\frac{d^s w(x)}{dx^s} = \sum_{j=0}^s w_j^s(x)\, H^{(j)}(z(x)), \qquad 0 \le s \le m,
$$
where, away from $x=0$,  the coefficients $w_j^s(x)$ are defined recursively by
$$
w_j^{s+1}(x) =
\begin{cases}
(w_0^s)'(x), & j = 0, \\
(w_j^s)'(x) + z'(x)\, w_{j-1}^s(x), & 1 \le j \le s, \\
z'(x)\, w_s^s(x), & j = s+1,
\end{cases}
\qquad \text{with } w_0^0(x) = 1.
$$
In particular, we have $w_s^s(x) = z'(x)^s$, for all $s \ge 0$. 

We now express the derivatives of $w$ in matrix form:
$$
\begin{pmatrix}
w(x) \\
w^{(1)}(x) \\
\vdots \\
w^{(m)}(x)
\end{pmatrix}
=
A(x)
\begin{pmatrix}
H(z(x)) \\
H^{(1)}(z(x)) \\
\vdots \\
H^{(m)}(z(x))
\end{pmatrix}, \qquad
A(x) = 
\begin{pmatrix}
w_0^0(x) &  & & \\
w_0^1(x) & w_1^1(x) & & \\
\vdots & \vdots & \ddots & \\
w_0^m(x) & w_1^m(x) & \cdots & w_m^m(x)
\end{pmatrix},
$$
where $A(x)$ is a lower triangular matrix with diagonal entries $w_s^s(x) = z'(x)^s$. 
Evaluating the expression above at $x = 0^-$ and $x = 0^+$, i.e. writing the conditions for $q(x)-z(x) = w(x)$, and noting that $z(0^-) = -1$, $z(0^+) = 1$, we obtain the linear systems:
$$
\begin{pmatrix}
\alpha_0 - z(0^-) \\
\alpha_1 - z^{(1)}(0^-) \\
\vdots \\
\alpha_m - z^{(m)}(0^-)
\end{pmatrix}
= A(0^-)
\begin{pmatrix}
H(-1) \\
H^{(1)}(-1) \\
\vdots \\
H^{(m)}(-1)
\end{pmatrix}, \qquad
\begin{pmatrix}
\beta_0 - z(0^+) \\
\beta_1 - z^{(1)}(0^+) \\
\vdots \\
\beta_m - z^{(m)}(0^+)
\end{pmatrix}
= A(0^+)
\begin{pmatrix}
H(1) \\
H^{(1)}(1) \\
\vdots \\
H^{(m)}(1)
\end{pmatrix}.
$$
Since the diagonal entries of $A(0^\pm)$ are powers of $z'(0^\pm) \neq 0$, the matrices are invertible, and the values of $H$ and its derivatives at $y = \pm 1$ can be uniquely determined.

By Proposition~\ref{prop:Hermite2}, there exists a trigonometric polynomial $H$ of degree $2(m+1)$ on $[-1,1]$ such that $H^{(s)}(-1)$ and $H^{(s)}(1)$ match these prescribed values for all $0 \le s \le m$. Then the function $q(x) := z(x) + H(z(x))$ satisfies \eqref{eq:jump_Q}, as required.
\end{proof}

With this result, we are now ready to prove Theorem~\ref{thm:main}.

\medskip
\noindent
{\bf Proof of Theorem \ref{thm:main}.} 
We construct the approximation and establish the error estimate in six steps, as follows.

\smallskip
\noindent
\textbf{Step 1.} {\it Construction of a composite function with prescribed jump discontinuities.} 
Let \( z(x) := \sgn(x) + \sin(x) \). Then $z$ belongs to \( {C_{\mathrm{cl}}^m}((-1,1) \setminus \{0\}) \) and satisfies
\[
z(0^-) = -1, \quad z(0^+) = 1, \quad z'(0^-) = z'(0^+) = 1 \neq 0.
\]
By Lemma~\ref{lem:axuiliary}, there exists a trigonometric polynomial \( H \) with \( 2(m+1) \) terms such that the composite function \( q(x) := z(x) + H(z(x)) \) satisfies the jump conditions \eqref{eq:jump_Q} for any prescribed values \( \alpha_s, \beta_s \in \mathbb{R} \) with \( 0 \le s \le m \). Setting \( \alpha_s = f^{(s)}(0^-) \) and \( \beta_s = f^{(s)}(0^+) \), we ensure that \( q \) replicates the jumps of \( f \) and its derivatives up to order \( m \) at the origin.
Since $H$
only depends on $f$ and $m$, its coefficients can be bounded as
$|c_{2k+1}|\leq d_m$, where $d_m$
depends on $f$ and $m$.

\smallskip
\noindent
\textbf{Step 2.} {\it Construction of a smooth residual function.}
Define the residual function
\[
r(x) := f(x) - q(x).
\]
Since \( f, q \in {C_{\mathrm{cl}}^m}((-1,1) \setminus \{0\}) \) and \( q \) is constructed to match the jumps of \( f \) at \( x = 0 \), the difference \( r \) is globally smooth: \( r \in C_{\rm b}^m((-1,1)) \).

\smallskip
\noindent
\textbf{Step 3.} {\it Approximation of the residual by a shallow Fourier network.} By Theorem~\ref{thm:nonuniformFourier}, there exists a shallow (single-layer) Fourier network \( R_W \) of the form
\[
R_W(x) = \sum_{k=1}^{N} \hat{r}_k e^{i\omega_k x}, \qquad N = W + 2(m+1),
\]
such that
\begin{equation}\label{eq:bound_R}
\| r - R_W \|_{L^2(-1,1)} \le C_{R,m} W^{-m+1/2},
\end{equation}
where \( C_{R,m} > 0 \) depends on $f$ and $m$ (via $r$ and $q$), but not on $W$. 

Moreover, the coefficients of \( R_W \) are bounded by a constant \( D_{R,m} \), which depends on $f$ and $m$, but not on $W$. The frequencies $\{\omega_k\}$ are independent of $f$.

\smallskip
\noindent
\textbf{Step 4.} {\it Approximation of the discontinuous composite function by a deep network.} 
Let \( S_L \) be the depth-\( L \), width-one Fourier network from Theorem~\ref{thm:step} satisfying
\begin{equation}\label{eq:bound_S_refined}
\| \sgn - S_L \|_{L^2(-1,1)} \le C_2\, 2^{-L/2}
\end{equation}
for some universal constant \( C_2 > 0 \). Define \( Z_L(x) := S_L(x) + \sin(x) \). Then \( Z_L \) approximates \( z(x) = \sgn(x) + \sin(x) \) with the same exponential rate. Consequently, the function $Z_L + H(Z_L)$ defines a depth-$(L+1)$ network approximating $q = z + H(z)$. Moreover, the definition in \eqref{eq:f0def} and \eqref{eq:fndef}
shows that all coefficients of
$S_L$, and therefore $Z_L$, are bounded by one.

\smallskip
\noindent
\textbf{Step 5.} {\it Final network construction.} 
We define the final approximation by
\[
F_{\rm FN}(x) := Z_L(x) + H(Z_L(x)) + R_W(x).
\]
This corresponds to a Fourier network with \( L + 1 \) layers, as illustrated in Figure~\ref{fig:thm3}. The first \( L \) layers compute the approximation \( S_L(x) \) to the sign function. In parallel, a single-neuron upper branch is added in layer \( L \) to compute \( \sin(x) \). These are combined to form the output \( Z_L(x) = S_L(x) + \sin(x) \). The final, \((L+1)\)-th, layer consists of two branches: an upper branch that takes \( x \) as input and computes \( R_W(x) \) using \( W + 2(m+1) \) neurons, and a lower branch that takes \( Z_L(x) \) as input and computes \( H(Z_L(x)) \) using \( 2(m+1) \) neurons. The outputs of these branches are added to $Z_L(x)$ to complete the approximation \( F_{\rm FN} \). 
The total number of neurons used is: $L \text{ (for } S_L) + 1 \text{ (for } \sin(x)) + (W + 2(m+1)) + 2(m+1) = L + W + 1 + 4(m+1)$.
\begin{figure}[!ht]
    \centering
    \includegraphics[width=0.99\linewidth]{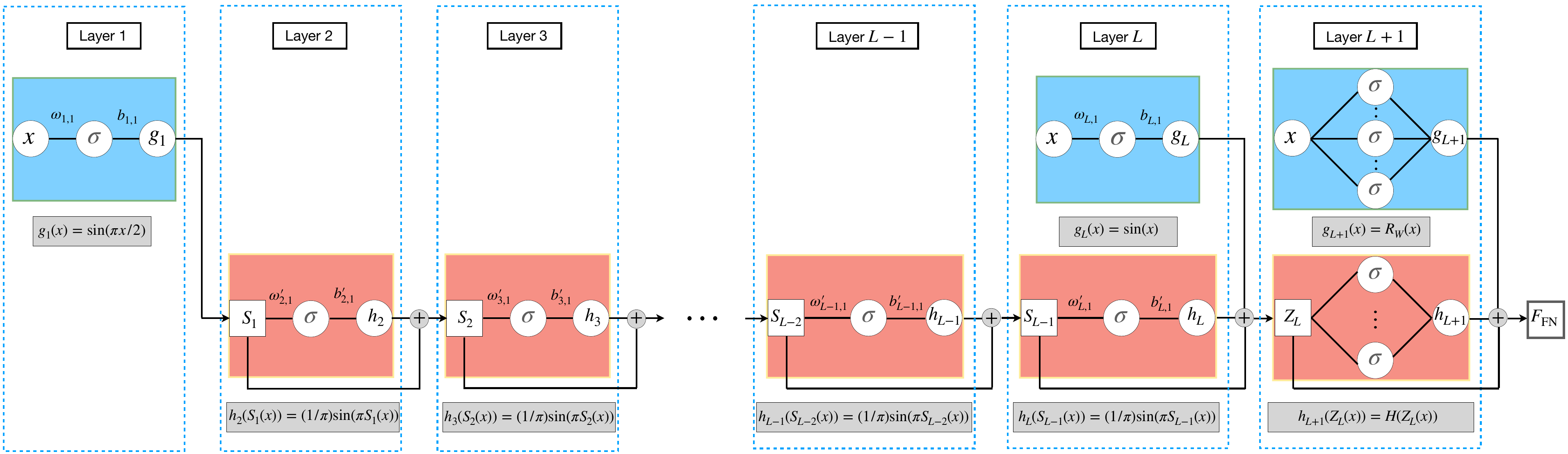}
    \caption{Schematic of the Fourier network \( F_{\rm FN} \) used to approximate \( f \). The first \( L \) layers generate \( S_L(x) \); a parallel branch at layer $L$ computes \( \sin(x) \), and their sum yields \( Z_L(x) \). The final layer consists of two branches: the upper branch computes \( R_W(x) \) from \( x \) using \( W + 2(m+1) \) neurons, while the lower branch computes \( H(Z_L(x)) \) from \( Z_L(x) \) using \( 2(m+1) \) neurons.  These two outputs are added to \( Z_L(x) \) to produce the final output \( F_{\rm FN}(x) \).}
    \label{fig:thm3}
\end{figure}

\smallskip
\noindent
\textbf{Step 6.} {\it Error estimate.} 
Using the triangle inequality, we estimate the total approximation error as
\begin{align*}
\| f - F_{\rm FN} \|_{L^2(-1,1)} 
&= \| q + r - (Z_L + H(Z_L) + R_W) \|_{L^2(-1,1)} \\
&= \| (z + H(z)) - (Z_L + H(Z_L)) + (r - R_W) \|_{L^2(-1,1)} \\
&\le \| z - Z_L \|_{L^2(-1,1)} + \| H(z) - H(Z_L) \|_{L^2(-1,1)} + \| r - R_W \|_{L^2(-1,1)}.
\end{align*}
Since $H$ is a trigonometric polynomial, it is Lipschitz on bounded intervals. Therefore,
\[
\| H(z) - H(Z_L) \|_{L^2(-1,1)} \le C_H \| z - Z_L \|_{L^2(-1,1)},
\]
for some constant \( C_H > 0 \)
that depends on $f$ and $m$, since $H$ depends on $f$ and $m$. Using this,
the fact that $z-Z_L=\sgn-S_L$ and the bounds \eqref{eq:bound_R},\eqref{eq:bound_S_refined}, we obtain
\[
\| f - F_{\rm FN} \|_{L^2(-1,1)} \le 
(1+C_H)
\| \sgn-S_L \|_{L^2(-1,1)} +  \| r - R_W \|_{L^2(-1,1)}\leq
C_m \left( 2^{-L/2} + W^{-m+1/2} \right),
\]
where \( C_m=(1+C_H)(C_2+C_{R,m}) > 0 \) depends on $f$ and $m$ but not
on $W$ or $L$.

The coefficients
of $F_{\rm FN}$
are all bounded by $\max(d_m,d'_m,1)$,
independent of $W$ and $L$.
Its frequencies are the union
of the frequencies in $S_L$,
$\sin(x)$, $H$ and $R_W$,
which are all independent of $f$.
This is obvious for $\sin(x)$.
For $R_W$ it is given in
Theorem~\ref{thm:nonuniformFourier}.
By \eqref{eq:f0def}, \eqref{eq:fndef}
the frequencies in $S_L$
are simply $\pi$ in each level except
in the first, where it is $\pi/2$,
for all $f \in {C_{\mathrm{cl}}^m}((-1,1) \setminus \{ 0 \}) $.
Finally, the trigonometric polynomial $H$
is of the form \eqref{eq:Hform}
and thus consists of the frequencies
$\{(2k+1)\pi/4$ for $k=-m-1,\ldots, m$,
also independent of $f$.
This completes the proof. 
\qed

%%%%%%%%%%%%%%%%%%%%%%%%%%%%%%%%%%%%%
\section{Numerical Experiments}
\label{sec:numerics}

We now present numerical experiments that verify the theoretical results established in Sections~\ref{sec:step} and~\ref{sec:generalization}. We also demonstrate that similar spectral convergence can be achieved in practice by replacing the constructive approximation with a training-based approach using the optimal sampling algorithm developed in~\cite{Davis_etal:2025}.

\subsection{Constructive approximation of the sign function}

We revisit the motivating example from Section~\ref{sec:step} and compare the performance of the Fourier ResNet with that of a truncated Fourier series in approximating the sign function. Convergence plots demonstrate the exponential decay of the $L^1$ error achieved by the ResNet construction.

Specifically, we consider the Fourier ResNet defined in equations~\eqref{eq:f0def}--\eqref{eq:fndef} with width $W = 1$ and depth $L \ge 2$, and compare it against the $L$-term Fourier sine series
\[
\tilde{f}_L(x) := \sum_{\ell = 1}^L \frac{4}{\pi (2 \ell - 1)} \sin((2 \ell - 1) x), \qquad x \in [-1,1],
\]
which approximates the sign function using the same number of terms, and thus comparable computational cost.

Figure~\ref{fig:sign_approx} shows both approximations for $L = 5$ and $L = 20$. As expected, the truncated Fourier series suffers from Gibbs oscillations near the discontinuity, while the ResNet approximation remains monotonic and entirely eliminates the undershoots and overshoots. 
\begin{figure}[htbp]
    \centering
\includegraphics[width=0.48\textwidth]{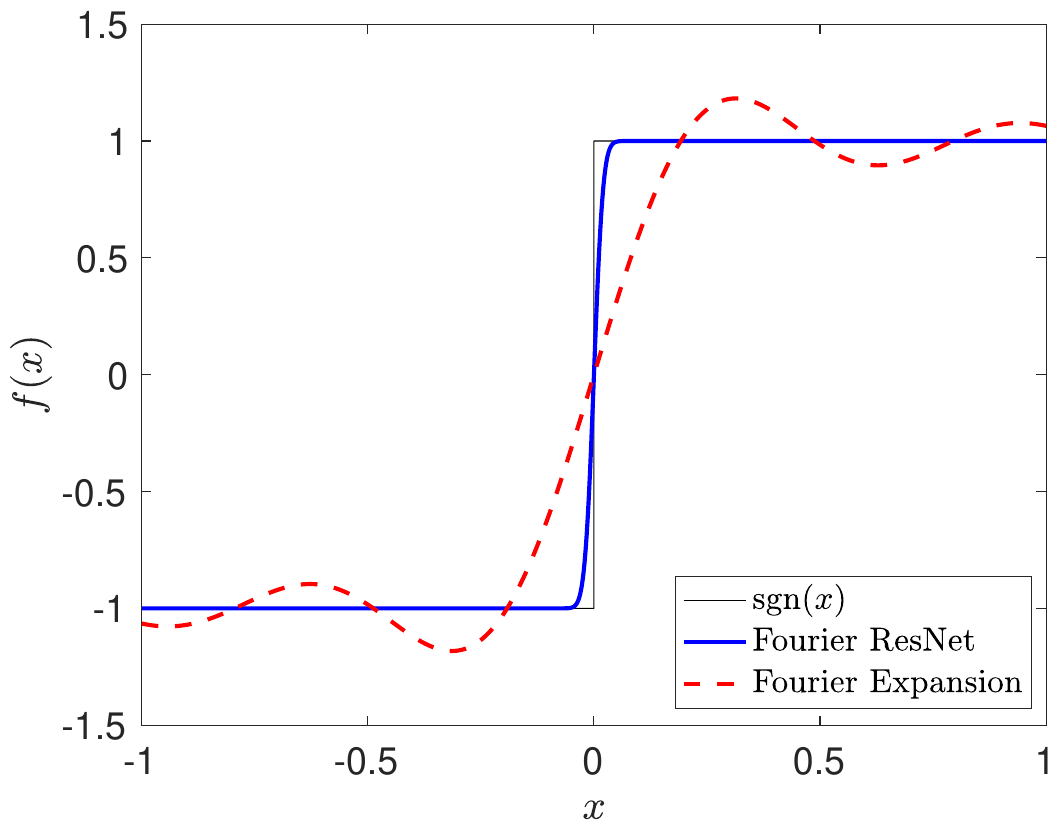}
    \hfill
\includegraphics[width=0.48\textwidth]{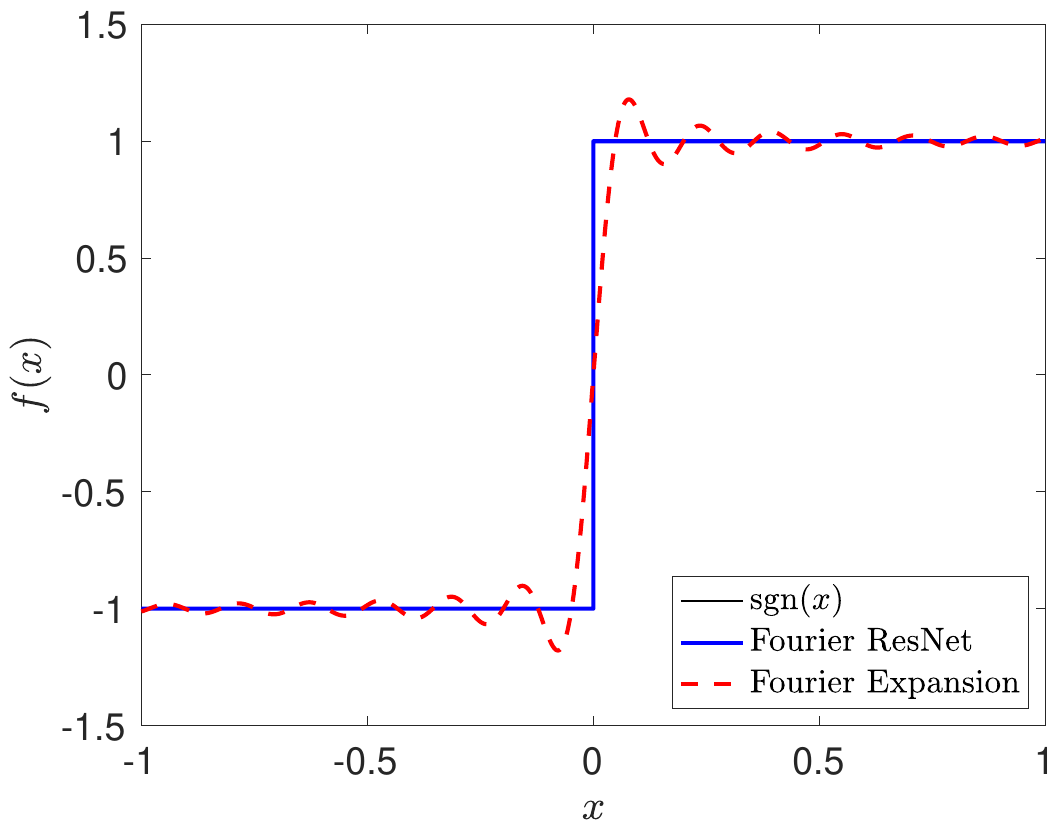}
\caption{Approximation of the sign function on $[-1,1]$ using the Fourier ResNet (solid blue) and the truncated Fourier series (dashed red) for $L = 5$ (left) and $L = 20$ (right). The exact sign function is shown as a thin solid black line for reference. While the truncated Fourier series exhibits Gibbs oscillations, the ResNet approximation remains monotonic and fully resolves the Gibbs phenomenon.}
    \label{fig:sign_approx}
\end{figure}

Figure~\ref{fig:sign_conv} displays the $L^1$ error as a function of $L$, in a log-linear scale. The plot confirms that the Fourier ResNet achieves exponential convergence in the $L^1$ norm, in stark contrast to the algebraic convergence of the truncated Fourier series.
\begin{figure}[htbp]
    \centering
\includegraphics[width=0.55\textwidth]{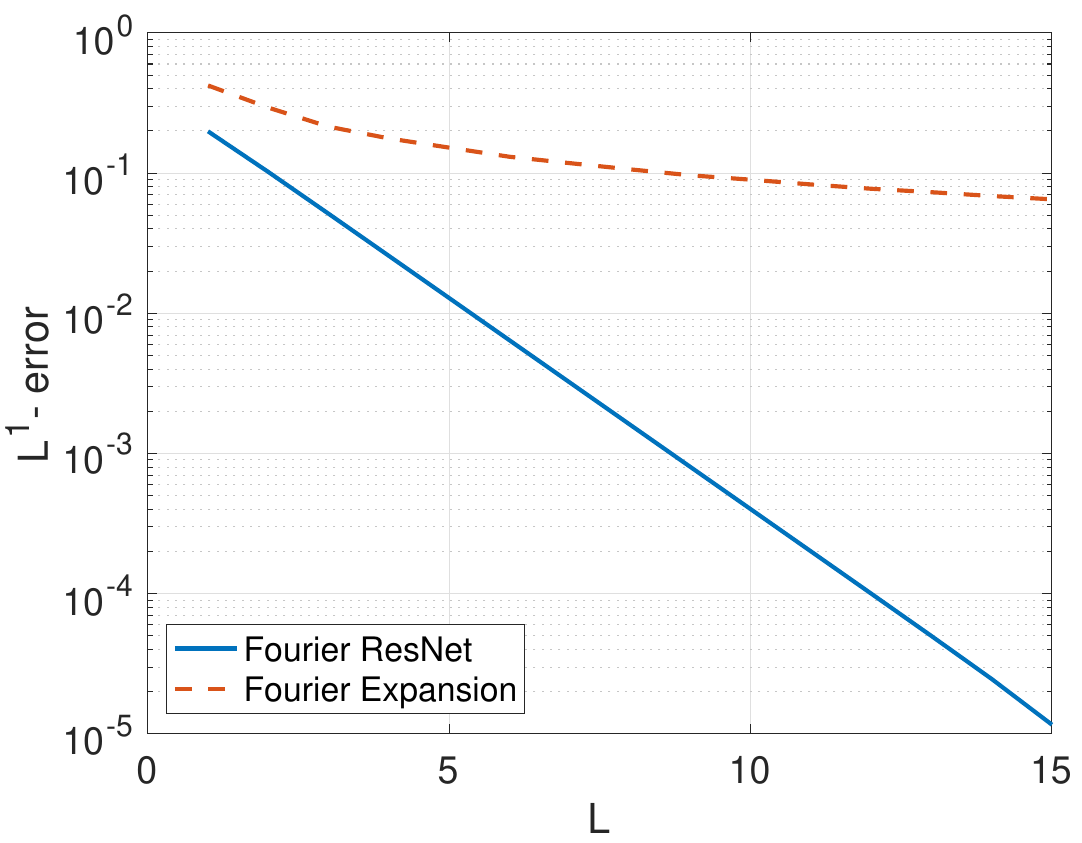}
    \caption{$L^1$ approximation error of the Fourier ResNet (solid) and the truncated Fourier series (dashed) for the sign function, plotted against the number of terms $L$ on a log-linear scale. The Fourier ResNet achieves exponential convergence, while the truncated Fourier series converges only algebraically.}
    \label{fig:sign_conv}
\end{figure}

\subsection{Constructive approximation of general piecewise smooth functions}

We now consider two representative examples of the more general class of piecewise-smooth functions discussed in Section~\ref{sec:generalization}.

\paragraph*{Piecewise-smooth function with jump discontinuities.} 
As the first example, we consider
\begin{equation}\label{eq:ex2a}
f(x) = \begin{cases}
1 + x, & x \in [-1, 0] \\
1 + \cos(\pi x), & x \in (0, 1]
\end{cases}, 
\end{equation}
as a prototype of functions that exhibit jump discontinuities in both their values and derivatives.

Figure~\ref{fig:piecewise-resnet-vs-fs} compares the approximation of $f$ using a standard truncated Fourier series and the proposed Fourier ResNet architecture. The left panel corresponds to $N = 20$ terms, with the ResNet configured using $m = 1$, $W = 6$, $L = 5$. The right panel uses $N = 30$, with $m = 1$, $W = 10$, and $L = 11$. In both cases, the total number of ResNet neurons is $N = L + 1 + W + 4(m+1)$, ensuring comparable computational cost between the two methods. In each plot, the target function $f$ is shown in thin solid black, the ResNet approximation in solid blue, and the truncated Fourier series in solid red. 
As expected, the truncated Fourier series suffers from pronounced Gibbs oscillations near the point of discontinuity at $x=0$. In contrast, the Fourier ResNet approximation exhibits much more localized oscillations, with the support of these oscillations diminishing rapidly as $N$ increases.
\begin{figure}[htb]
\centering
\includegraphics[width=0.45\textwidth]{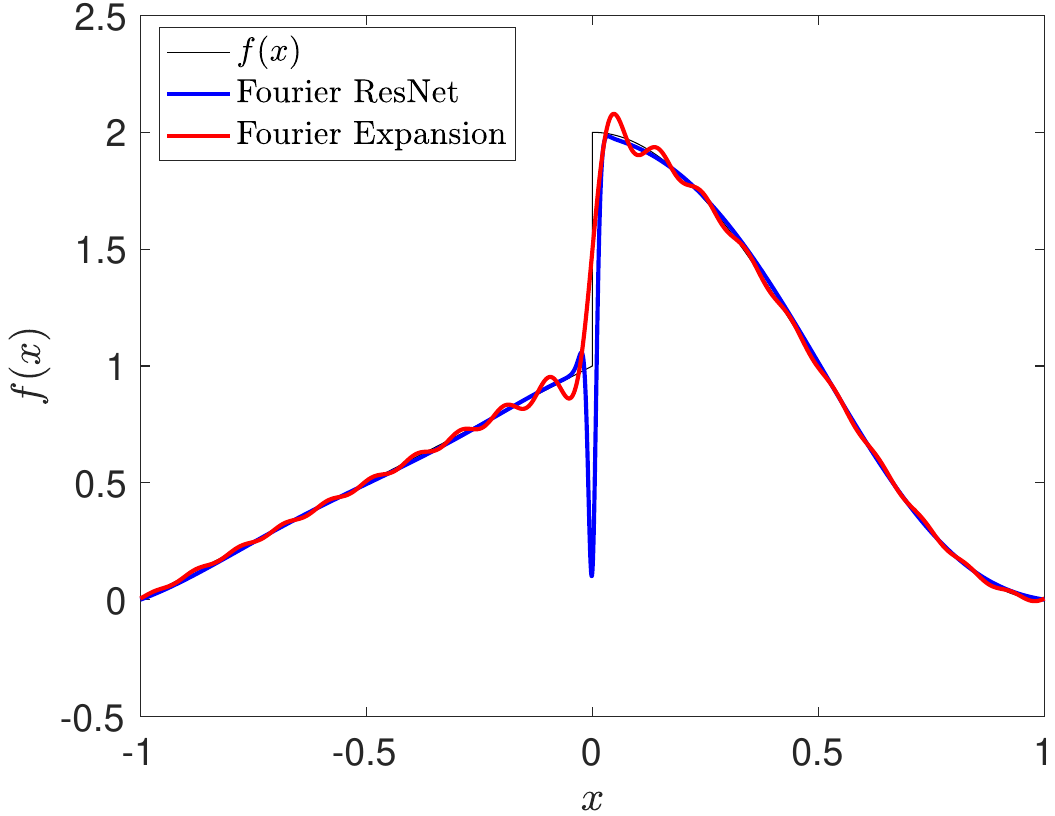}
\hfill
\includegraphics[width=0.45\textwidth]{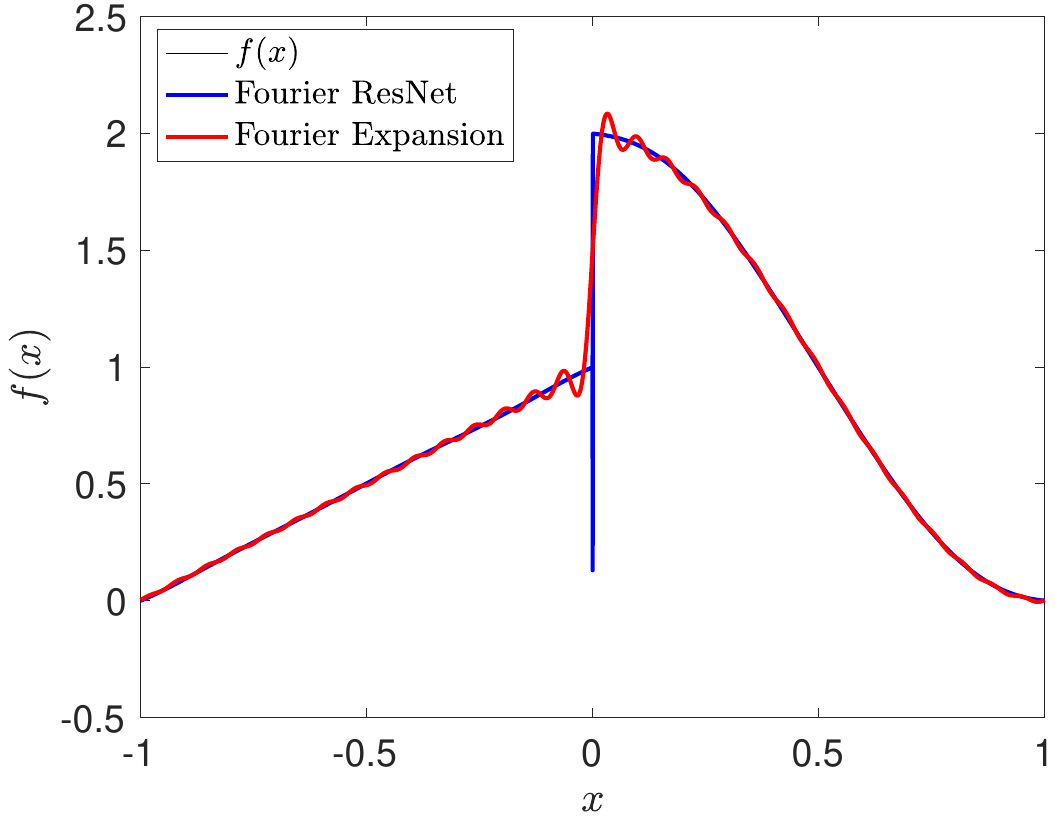}
\caption{Approximation of a piecewise smooth function using a truncated Fourier series (red) and a Fourier ResNet (blue), with both methods using the same total number of $N$ terms or neurons. Left: $N = 20$, $m = 1$, $W = 6$, $L = 5$. Right: $N = 30$, $m = 1$, $W = 10$, $L = 11$. The target function $f$ is shown in thin solid black. While the Fourier series exhibits Gibbs oscillations, the ResNet approximation localizes and suppresses these artifacts.}
\label{fig:piecewise-resnet-vs-fs}
\end{figure}

Figure~\ref{fig:ex2_Gibbs_decay} illustrates the rapid decay in the spatial support of spurious oscillations in the Fourier ResNet approximation as the depth parameter $L$ increases. We fix $m = 1$ and $W = 10$, and vary $L$ from 3 to 7, resulting in a sequence of five approximations with increasing total neuron count $N$ from 22 to 26. As expected, the widest oscillation support occurs for the smallest depth, $L = 3$, with progressively more localized behavior observed as $L$ increases. Even modest increases in depth are sufficient to significantly contract the oscillatory region near the point of non-smoothness at $x = 0$, and the trend continues with larger values of $L$.
\begin{figure}[htbp]
    \centering
\includegraphics[width=0.99\textwidth]{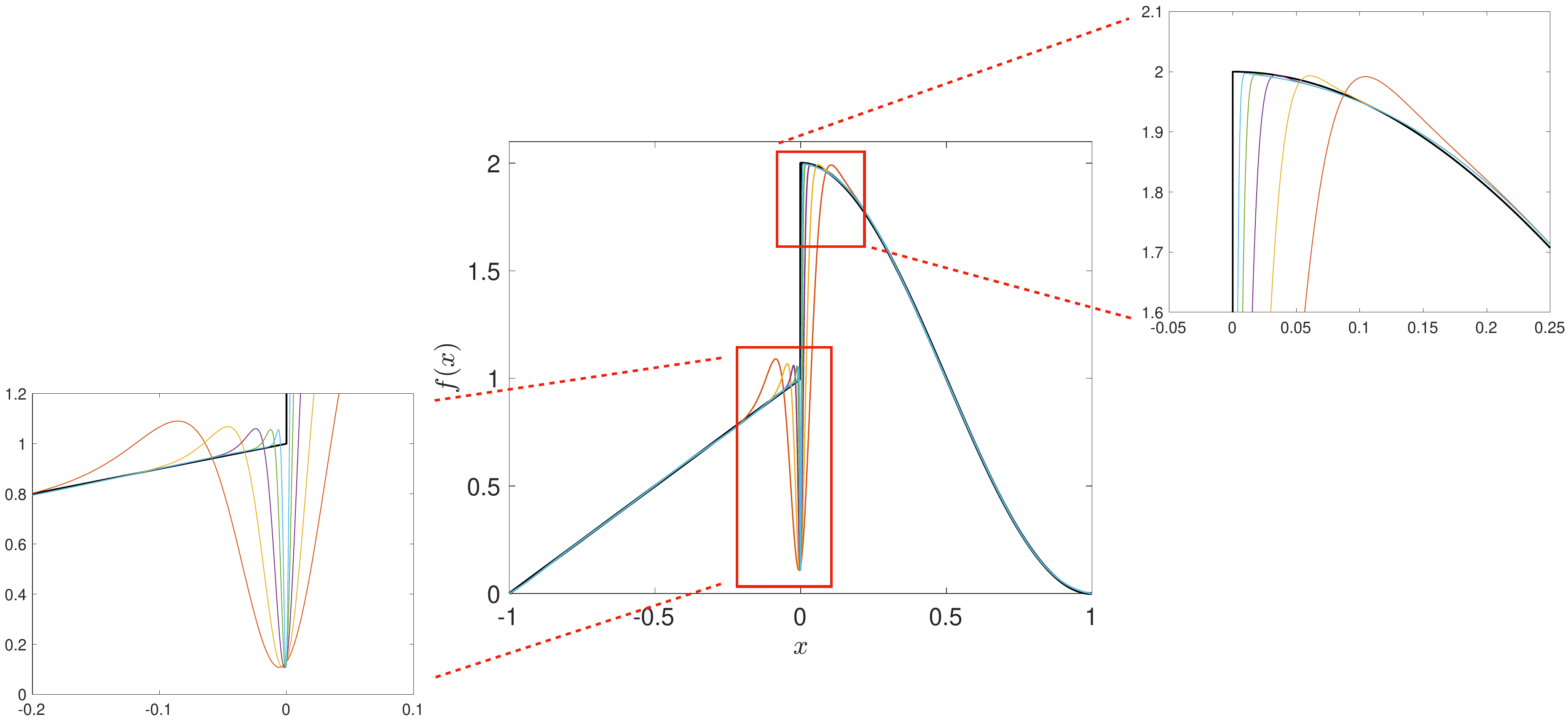}
    \caption{Decay in the spatial support of spurious oscillations in the Fourier ResNet approximation of the piecewise smooth function~\eqref{eq:ex2a}, for fixed $m = 1$ and $W = 10$. The central panel shows the approximation over the full domain, while two external zoomed panels highlight the behavior near $x = 0$, one to the left and one to the right of the origin. Each curve corresponds to a different depth $L \in \{3,4,5,6,7\}$, illustrating that the oscillations rapidly contract toward the singularity as $L$ increases.}
    \label{fig:ex2_Gibbs_decay}
\end{figure}

Figure~\ref{fig:ex2_conv} reports the $L^2$ approximation error of the Fourier ResNet as a function of the width parameter $W$, plotted on a log-log scale. The depth is fixed at $L = 20$ to ensure that the first term in the error bound \eqref{eq:main_estimate} becomes negligible, so that the observed error predominantly reflects the second term, which is theoretically expected to decay as ${\mathcal O}(W^{-m})$. The plot includes curves for $m \in \{1,2,3,4\}$ and $W = 10 \cdot 2^j$ for $j = 0, \dots, 6$. 
For comparison, we also include the $L^2$ error from approximating the same function using a standard truncated Fourier series with $N = 41 + W$ terms. This choice matches the parameter count of the Fourier ResNet with $m = 4$ and $L = 20$, ensuring comparable computational cost between the two methods.
\begin{figure}[htbp]
    \centering
\includegraphics[width=0.55\textwidth]{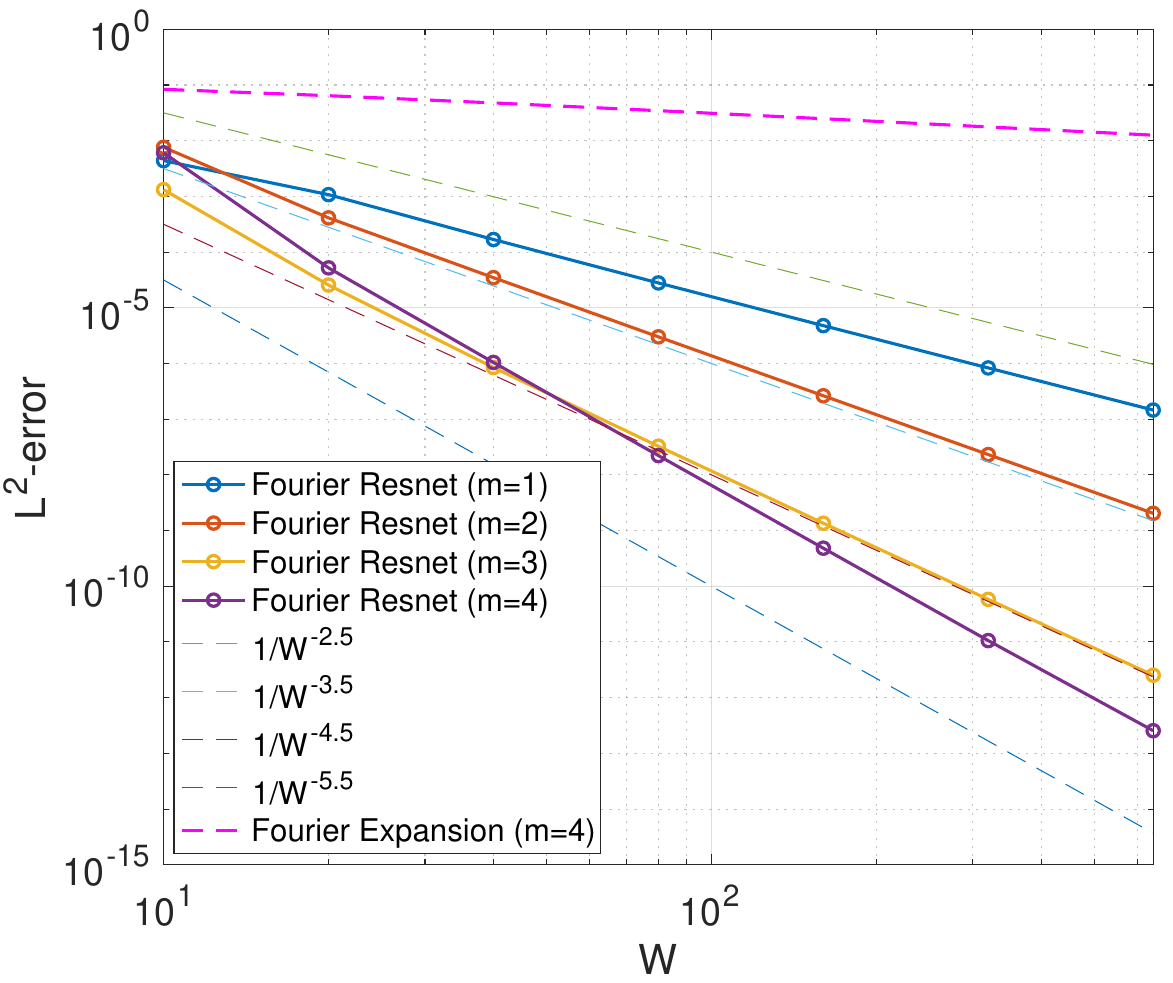}
    \caption{$L^2$ error of the Fourier ResNet approximation of the piecewise smooth function \eqref{eq:ex2a} versus the width parameter $W = 10 \cdot 2^j$ for $j = 0, \dots, 6$, shown on a log-log scale. The depth is fixed at $L = 20$ to isolate the second term in the error bound~\eqref{eq:main_estimate}. Curves correspond to $m \in \{1,2,3,4\}$. The results confirm spectral convergence of Fourier ResNet, with empirical rates close to ${\mathcal O}(W^{-m - 3/2})$, while the truncated Fourier series converges algebraically.}
    \label{fig:ex2_conv}
\end{figure}

As expected, the truncated Fourier series exhibits algebraic convergence, while the Fourier ResNet achieves spectral convergence as $W$ increases. Notably, the empirical decay appears slightly faster than the expected theoretical rate ${\mathcal O}(W^{-m})$, trending toward ${\mathcal O}(W^{-m - 3/2})$, which suggests that the practical performance of the network may exceed the theoretical worst-case bounds.

\paragraph*{Hat function with derivative discontinuity.} 
As a second example, we consider the hat function:
\begin{equation}\label{eq:ex2b}
f(x) = \begin{cases}
1 + x, & x \in [-1, 0] \\
1 - x, & x \in (0, 1]
\end{cases},
\end{equation}
which is continuous but not differentiable at $x = 0$.

As in the previous example, increasing $L$ leads to a rapid localization of oscillatory artifacts near the singularity at $x = 0$, exhibiting behavior similar to the earlier case. 
Figure~\ref{fig:hat_conv} shows the $L^2$ approximation error of the Fourier ResNet applied to the hat function, plotted on a log-log scale as a function of $W$. The network depth is fixed at $L = 20$ to isolate the influence of $W$ on the second term in the error bound~\eqref{eq:main_estimate}. The error curves correspond to smoothness indices $m \in {1, 2, 3, 4}$, with widths given by $W = 10 \cdot 2^j$ for $j = 0, \dots, 6$. The results confirm spectral convergence, consistent with theoretical expectations, and suggest empirical decay rates close to ${\mathcal O}(W^{-m - 3/2})$. For comparison, we also include the error from a standard truncated Fourier series with $N = 41 + W$ terms, which, as expected, shows algebraic convergence.
\begin{figure}[htbp]
    \centering
    \includegraphics[width=0.55\textwidth]{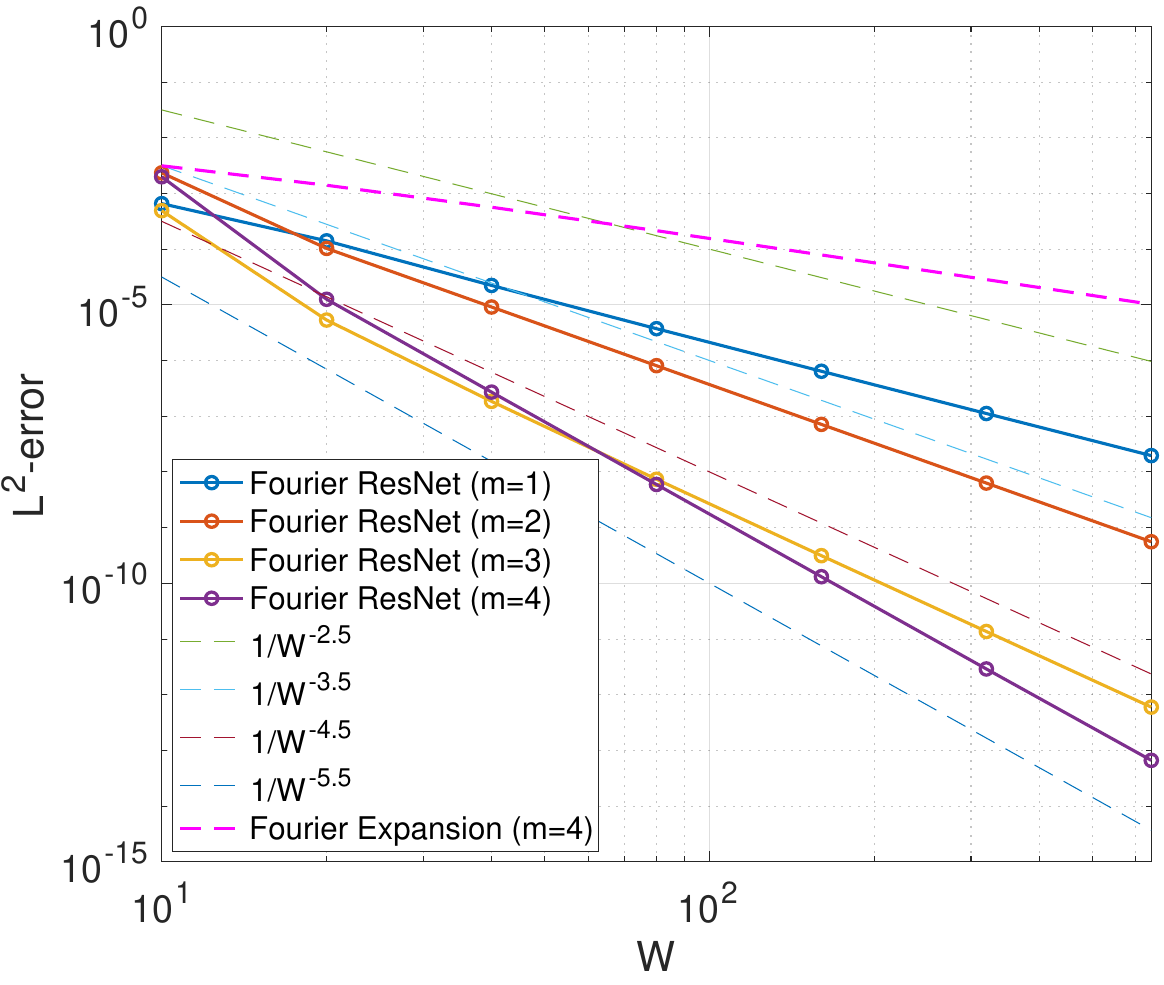}
    \caption{$L^2$ approximation error of the Fourier ResNet applied to the hat function \eqref{eq:ex2b}, plotted against the width parameter $W = 10 \cdot 2^j$ for $j = 0, \dots, 6$, on a log-log scale. The depth is fixed at $L = 20$, isolating the effect of $W$ on the second term of the error bound~\eqref{eq:main_estimate}. Curves for $m \in \{1,2,3,4\}$ confirm spectral convergence of Fourier ResNet, with empirical rates consistent with ${\mathcal O}(W^{-m - 3/2})$, while the truncated Fourier series converges algebraically.}
    \label{fig:hat_conv}
\end{figure}

This example further illustrates the ability of the Fourier ResNet to localize and suppress spurious oscillations in the presence of derivative discontinuities, achieving spectral accuracy with a relatively small number of parameters. Notably, the architecture remains effective even when the singularity differs qualitatively from the jump discontinuity in the previous example.

\subsection{Training-based approximation using the optimal sampling algorithm}\label{sec:sampling_numerics}

In this section, we investigate the performance of the optimal sampling-based training algorithm developed in \cite{Davis_etal:2025} in relation to the theoretical approximation rates established for shallow networks in Theorem~\ref{thm:nonuniformFourier} and deep networks in Theorem~\ref{thm:main}.

To test this, we approximate several target functions $f$ of varying regularity using Fourier residual networks $F_{W,L}$ trained with the optimal sampling algorithm. Throughout this section, $F_{W,L}$ denotes a Fourier ResNet with uniform width $W$ and depth $L$. We assess the quality of these approximations by computing the mean squared error (MSE) on a collection of samples $\{(x_{j}, f(x_j))\}_{j=1}^{M}$, defined as 
\begin{equation*}
    \text{MSE}(f, F_{W,L}) = \frac{1}{M}\sum_{j=1}^{M}(f(x_j) - F_{W,L}(x_j))^2.
\end{equation*}
Note that as $M\to \infty$, $\text{MSE}(f,F_{W,L})\to ||f-F_{W,L}||^{2}_{L^{2}}.$ Therefore, for sufficiently large $M$, the MSE decay rate is expected to follow the $L^{2}$-error decay rates with respect to network architecture established in Theorem \ref{thm:nonuniformFourier} and Theorem~\ref{thm:main}. 

We emphasize that the sampling algorithm constructs Fourier residual networks with fixed width and adaptive depth. Specifically, for a user-specified width $W$, layers of width $W$ are added and trained sequentially until a prescribed tolerance is met or a maximum depth is reached. Each newly added layer yields an improved approximation of the target function. Accordingly, in this section, the approximation rate with respect to depth $L$ is assessed by tracking the error after each successive layer within a single network. In contrast, because the width is not adapted during training, studying approximation rates with respect to $W$ requires training separate networks with progressively larger widths.

We present numerical results that demonstrate the algorithm's comparable performance to the theoretically constructed networks in certain cases. In other cases, the findings suggest directions for future work in the form of enhancements to the optimal sampling algorithm and opportunities for expanded theoretical results.

\paragraph*{Smooth non-periodic function.} 
As the first example, we consider the smooth non-periodic function
\begin{equation*}
    f(x) = \exp(-x^2/2)\cos(8x) + x, \qquad x\in [-1,1].
\end{equation*}
The aim of this example is to assess the MSE decay rate for shallow Fourier ResNets with respect to width parameter $W$ in comparison to the theoretically predicted rate of $\mathcal{O}(W^{-2m +1})$ from Theorem~\ref{thm:nonuniformFourier}. Using the optimal sampling algorithm, we train shallow Fourier ResNets of width $W+2m+1$ for $W\in\{2,4,8,16,32\}$ and $m\in\{1,2,3,4\}$ using $5000$ equidistributed training samples. We then evaluate these trained networks on $10000$ equidistributed test samples.
\begin{figure}[!htb]
    \centering
    \includegraphics[width=0.55\linewidth]{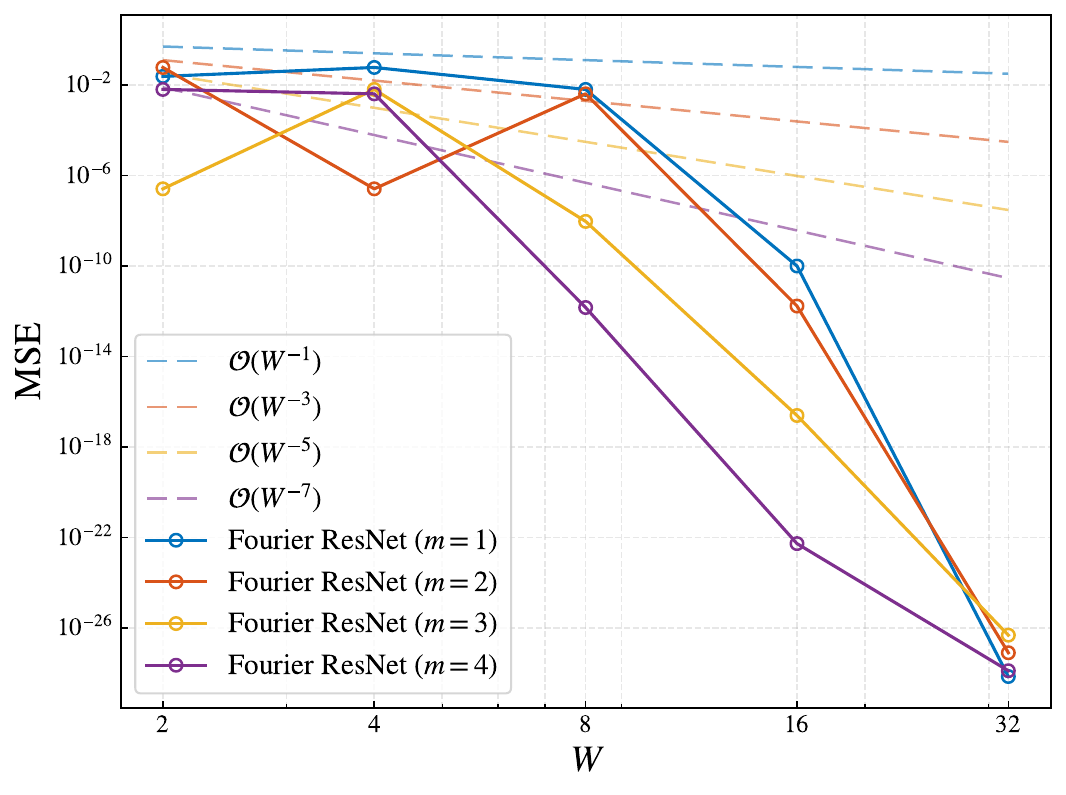}
    \caption{MSE in shallow Fourier ResNet approximations (solid circle) as a function of width parameter $W\in\{2,4,8,16,32\}$ plotted on a logarithmic scale; different colors distinguish different choices of smoothness index $m\in \{1,2,3,4\}$ with the dotted line of the same color representing the corresponding reference rate $\mathcal{O}(2^{-2m + 1})$.}
    \label{fig:rfnn_smooth_nonperiodic_function}
\end{figure}

In Figure~\ref{fig:rfnn_smooth_nonperiodic_function}, the solid lines with circular markers represent the MSE in the network approximations on the test samples as a function of $W$ plotted on a logarithmic scale. Different colors correspond to different smoothness indices $m$, with the dotted lines of the same color showing the corresponding reference rate $\mathcal{O}(W^{-2m+1})$ from Theorem~\ref{thm:nonuniformFourier}. For small $W$, the MSE in the network approximations does not exhibit monotonic decay over all choices of $m$. This is due to the inherent randomness in the sampling-based training algorithm. The frequency parameters of the network are approximately sampled from an optimal distribution that minimizes an upper bound on the $L^{2}$-error of the network; they are not necessarily the optimal pointwise choices with respect to minimizing the network's MSE. Hence for small $W$, nonmonotonicity may be expected. As $W$ grows larger, the MSE decay rate is faster than the theoretical reference rate for all choices of $m$, and the observed rate is similar across all smoothness indices. 

We remark here that Theorem~\ref{thm:nonuniformFourier} provides only an upper bound on the approximation rate and does not preclude faster rates for specific target functions, as observed in this example. The improved empirical performance may stem from the sampling-based training procedure, which approximately draws frequencies from a target-dependent optimal distribution, in contrast to the theoretical construction that relies on a fixed-point iteration with frequencies independent of the target function. For functions with relatively concentrated spectral content---such as the one considered here---this adaptive frequency sampling can yield particularly efficient and rapid approximation.

\paragraph*{The sign function.} 
As the second example, we consider the sign function $f(x)=\text{sgn}(x)$ as defined in \eqref{eqn:sgn_function},
which is piecewise constant with a jump discontinuity at $x=0$. The goal of this example is to investigate the MSE decay rate with respect to network depth $L$ in comparison to the theoretical rate of $\mathcal{O}(2^{-L})$ predicted by Theorem~\ref{thm:main}. This example also offers a qualitative comparison between Fourier ResNets trained with the optimal sampling algorithm and the constructive approximation of the sign function in Figure~\ref{fig:sign_approx}.

We train a Fourier ResNet of fixed width $W=2$ with $7000$ equidistributed training samples, and we increment the depth $L$ of the network until an MSE of $10^{-15}$ is achieved on this traing data. We achieve this machine precision at layer $L=17$. We then evaluate the trained network on $10000$ equidistributed test samples.
\begin{figure}[!htb]
    \centering
    \includegraphics[width=\linewidth]{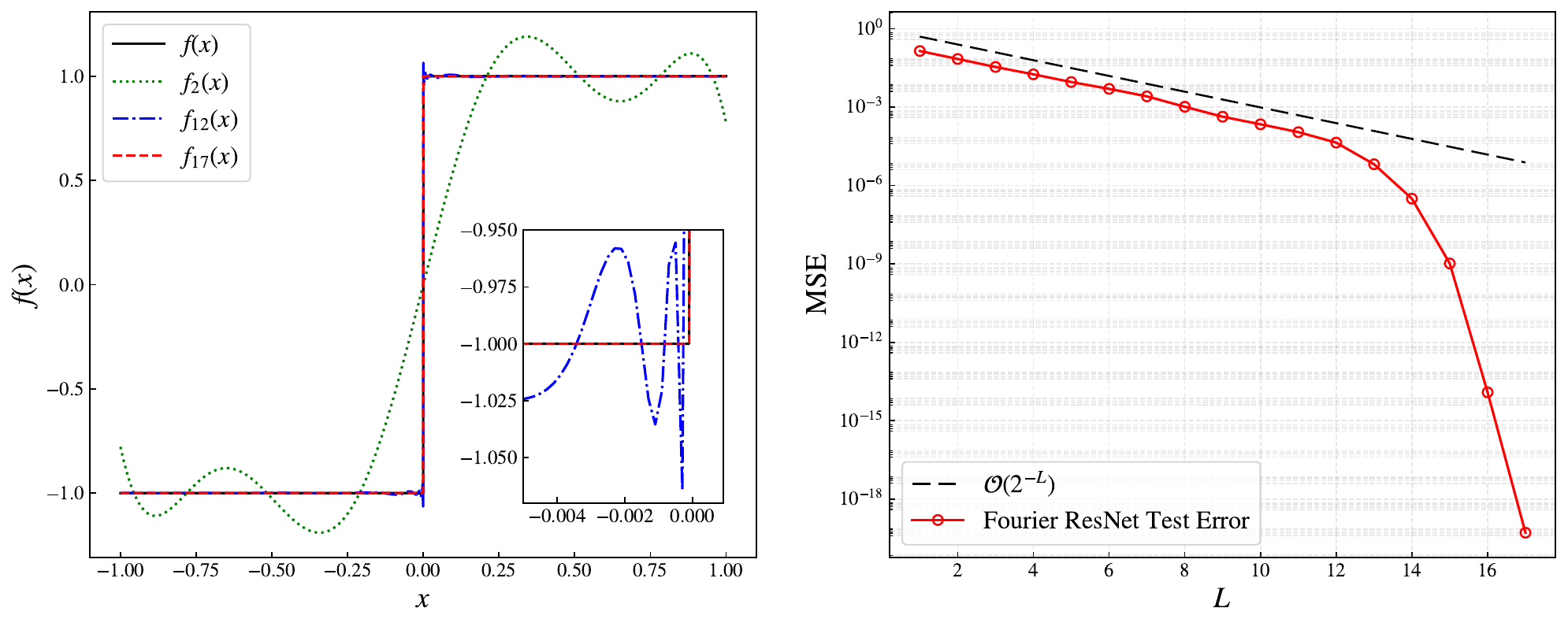}
    \caption{On the left, the true target function (black solid) along with approximations from a Fourier ResNet with fixed width $W=2$ after layers $L=2$ (green dot), $L=12$ (blue dash-dot), and $L=17$ (red dash); on the right, MSE in the Fourier ResNet approximation (red circle) as a function of the network depth $L$ together with the reference rate (black dash) $\mathcal{O}(2^{-L})$ plotted on a log-linear scale.}
    \label{fig:rfnn_sign_function}
\end{figure}

The left plot of Figure~\ref{fig:rfnn_sign_function} shows the true sign function (black solid) together with network predictions after layers $L=2$ (green dot), $L=12$ (blue dash-dot) and $L = 17$ (red dash) on the test data. The inset axes offer an enhanced view of the network approximations near the discontinuity. Unlike the theoretically constructed network in Figure~\ref{fig:sign_approx}, which exhibits a monotonic approximation of the sign function, we \emph{do observe overshoot} at intermediate steps of the training process; see e.g., the $L=12$ prediction (blue dash-dot) in the inset axes. Nevertheless, the final prediction of the network at layer $L=17$ (red dash) approximates the sign function to machine precision with \emph{no overshoot}. 

The right plot of Figure~\ref{fig:rfnn_sign_function} shows the MSE in the network approximation (red circle) on the test data as a function of the network depth $L$ together with the reference rate $\mathcal{O}(2^{-L})$ (black dash) plotted on a log-linear scale. For small values of $L$, the observed approximation rate closely matches the theoretical rate predicted by Theorem~\ref{thm:main}. As $L$ increases, however, we observe super-exponential convergence faster than the theoretical rate. The initial agreement between theory and experiment may be attributed to the close correspondence between the theoretically constructed network in Figure~\ref{fig:sign_approx}, which uses $W=1$, and the present approximation, which uses $W=2$. We hypothesize that this architectural similarity biases the sampling algorithm toward representations resembling the theoretical construction. The faster, super-exponential convergence observed as the error approaches machine precision may instead result from the sampling algorithm identifying more optimal frequencies than those used in the theoretical construction.

\paragraph*{Piecewise-smooth function with jump discontinuities.} 
As the third example, we consider the function $f(x)$ defined in \eqref{eq:ex2a},
which is a general piecewise-smooth function with jump discontinuities in both its value and derivative. This example assesses the decay rate in the optimal sampling algorithm trained Fourier ResNet MSE with respect to network depth $L$, and also provides a qualitative comparison with the theoretically constructed network in Figure~\ref{fig:ex2_Gibbs_decay}.

\begin{figure}[!htb]
    \centering
    \includegraphics[width=\linewidth]{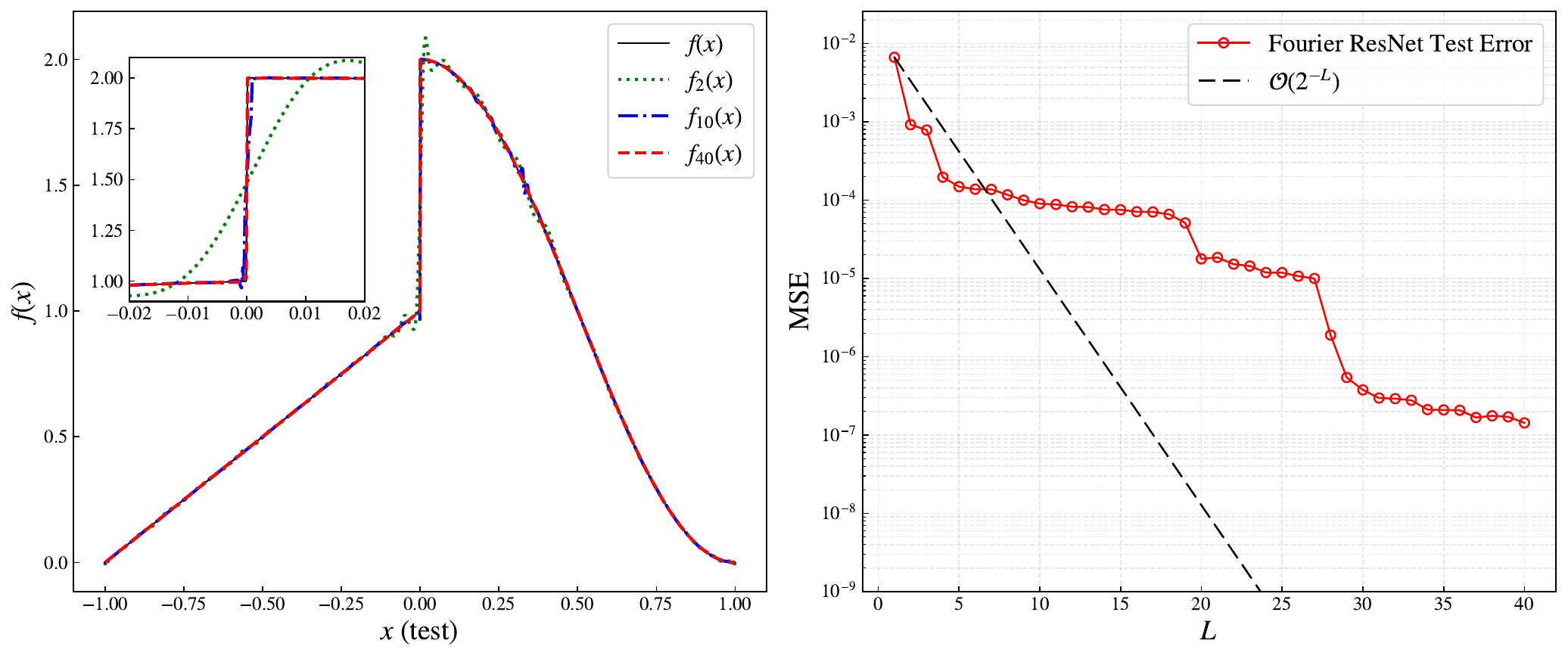}
    \caption{On the left, predictions on the test data from a Fourier ResNet of fixed width $W=40$ after layers $L=2$ (green dot), $L=10$ (blue dash-dot), and $L=40$ (red dash) together with the true target function (black solid); on the right, the MSE in the network approximation on the testing data (red circle) as a function of the network depth $L$ along with the reference rate $\mathcal{O}(2^{-L})$ (black dash) all plotted on a log-linear scale.}
    \label{fig:rfnn_gen_discont_function}
\end{figure}

We train a Fourier ResNet of fixed width $W=40$ and depth $L=40$ using $7000$ equidistributed training samples, and we evaluate the trained network on $10000$ equidistributed test samples.

Figure~\ref{fig:rfnn_gen_discont_function} presents both qualitative and quantitative results with respect to this approximation. The left plot shows the network prediction on the testing data after layers $L=2$ (green dot), $L=10$ (blue dash-dot) and $L=40$ (red dash) together with the true target function in black. The inset axes offer an enhanced view of the discontinuous interface.

The layer 2 prediction exhibits an overly-smoothed approximation of the discontinuity with overshoot and oscillatory errors localized near $x=0.4$. Away from these high-error regions the approximation closely matches the true target function. The predictions from layers $L=10$ and $L=40$ offer successively better approximations in both these high-error regions. The $L=10$ prediction exhibits a much sharper approximation of the discontinuity with very small overshoot on the left side and no visible overshoot on the right side. In the vicinity of $x=0.4$ we observe high-frequency low-magnitude oscillatory errors. The $L=40$ prediction is high-quality, closely matching the true target function throughout the domain. The discontinuity is resolved sharply with no visible overshoot, and the oscillatory errors near $x=0.4$ present in the approximations from the previous layers are no longer apparent.

The approximation obtained through sampling-based training is qualitatively similar to the constructed network in Figure~\ref{fig:ex2_Gibbs_decay}, but there are also notable differences. For small $L$, both predictions exhibit oscillatory overshoot near the discontinuous interface, and the support of these oscillations shrinks rapidly with increasing depth. In the constructed network, the overshoot persists even as this support contracts. By contrast, the network trained via the sampling algorithm shows no overshoot at $L=40$. On the other hand, away from the discontinuous interface, the theoretically constructed network is nearly exact even for small $L$, whereas the sampling-trained network exhibits small oscillatory errors near $x=0.4$ that decay slowly with depth.

The right plot of Figure~\ref{fig:rfnn_gen_discont_function} shows the MSE in the network approximation on the test data (red circle) as a function of network depth $L$ together with the reference rate $\mathcal{O}(2^{-L})$ (black dash) plotted on a log-linear scale. In contrast to the previous example, the observed approximation rate is slower than the theoretical rate predicted by Theorem~\ref{thm:main}. Several factors may contribute to this discrepancy.

First, as explained in \cite{Davis_etal:2025}, the optimal sampling algorithm does not explicitly aim to minimize the network's $L^{2}$-error. Instead, it samples the frequency parameters in each block based on optimal distributions that minimize an upper bound on the $L^{2}$-error specific to that block. Importantly, this upper bound is not sharp, which suggests that an optimization procedure explicitly targeting the minimization of $L^{2}$-error may be necessary to achieve the theoretical convergence rate for these more complex discontinuous functions.

Second, the optimal sampling algorithm aims to distribute the frequency parameters of the network according to an analytic optimal distribution. This inevitably results in near-duplicated frequencies in regions of high-probability. This can cause ill-conditioning in the least squares problem for the corresponding amplitude parameters, especially when the width $W$ is moderate to large, as is the case here. This ill-conditioning is primarily addressed through a Tikhonov regularization on the least squares problem. This is effective when only moderate tolerances are desired, but an approximation that converges to machine precision generally requires a commensurate reduction in the value of the Tikhonov constant, reintroducing conditioning issues. In practice, we hypothesize that this ill-conditioning could result in slowed convergence behavior. 

Third, the theoretical insights presented in this work indicate that efficient approximation of discontinuous functions can be achieved by utilizing compositions of Fourier modes (network depth) to approximate the discontinuous interface, while employing standard Fourier modes for the smooth regions of the function away from the discontinuity. This division of labor is not inherently incorporated into the sampling-based training algorithm.

The observations in this example motivate future theoretical and numerical research into Fourier ResNet approximations of general discontinuous functions. From a theoretical perspective, these results motivate the discovery of constructions that completely avoid overshoot at the discontinuous interface, and from a numerical perspective, they spur the development of new training strategies that can realize the theoretical convergence rate, perhaps through carefully designed network architectures inspired by the theoretical constructions. 

\section{Conclusions}
\label{sec:conclusions}

In this work, we developed a constructive approximation framework showing that Fourier residual networks can accurately represent functions with jump discontinuities while retaining high-order approximation in smooth regions. In particular, we established algebraic convergence rates determined by the available smoothness and showed that, for piecewise-$C^\infty$ functions, the approximation becomes spectral in the width parameter and exponential in depth. The construction provides a mechanism for resolving the Gibbs phenomenon through increasingly localized oscillatory regions, offering an alternative to classical spectral reconstruction techniques.

Several directions for future work naturally arise. On the theoretical side, it would be of interest to compare the present framework with advanced spectral reconstruction methods, such as filtering and Gegenbauer-based approaches, and to further analyze the behavior of the Gibbs phenomenon in the network setting, including the structure and scaling of the resulting oscillatory artifacts. Another important direction is the extension of the present analysis to higher-dimensional functions, where the geometry of discontinuities and the choice of frequency distributions become significantly more complex. From a computational perspective, it remains an open question to what extent the constructive mechanisms identified here can be effectively realized through training algorithms, and whether randomized or sampling-based strategies can achieve similar approximation behavior in practice.

%%%%%%%%%%%%%%%%%%%%%%%%%%%%%%%%%%%%%%%%%%%%%%%%%
\section*{Statements and Declarations}

\subsection*{Conflict of Interest}
The authors declare that they have no conflict of interest.

\subsection*{Funding}

No funding was received to assist with the preparation of this manuscript.

\subsection*{Author Contributions}

All authors contributed meaningfully to the research and writing of the manuscript. OD led the algorithm development and numerical experiments. MM contributed to theoretical development and algorithm design. OR led the theoretical analysis and proofs.

\subsection*{Acknowledgements}
Sandia National Laboratories is a multi-mission laboratory managed and operated by National Technology \& Engineering Solutions of Sandia, LLC (NTESS), a wholly owned subsidiary of Honeywell International Inc., for the U.S. Department of Energy’s National Nuclear Security Administration (DOE/NNSA) under contract DE-NA0003525. This written work is authored by an employee of NTESS. The employee, not NTESS, owns the right, title and interest in and to the written work and is responsible for its contents. Any subjective views or opinions that might be expressed in the written work do not necessarily represent the views of the U.S. Government. The publisher acknowledges that the U.S. Government retains a non-exclusive, paid-up, irrevocable, world-wide license to publish or reproduce the published form of this written work or allow others to do so, for U.S. Government purposes. The DOE will provide public access to results of federally sponsored research in accordance with the DOE Public Access Plan.

\subsection*{Code/Data Availability} 
For the results in Section~\ref{sec:sampling_numerics}, no new code was developed. The implementation follows the method described in \cite{Davis_etal:2025}, where the algorithm is presented in pseudocode. The implementation used by the authors is not publicly available because external release is subject to Sandia National Laboratories institutional copyright and software release and review procedures.

% -------------------- REFERENCES ----------------
\bibliography{refs_motamed.bib}

@book {MC,
AUTHOR = {Fishman, G. S.},
     TITLE = {Monte {Carlo}: {C}oncepts, {A}lgorithms, and {A}pplications},
 PUBLISHER = {Springer- Verlag},
   ADDRESS = {New York},
      YEAR = {1996}
}

@article{KlusowskiBarron2018,
  author    = {J. M. Klusowski and A. R. Barron},
  title     = {Approximation by Combinations of {ReLU} and Squared {ReLU} {R}idge Functions with \( \ell_1 \) and \( \ell_0 \) Controls},
  journal   = {IEEE Transactions on Information Theory},
  volume    = {64},
  pages     = {7649--7656},
  year      = {2018}
}

@article{Hornik_etal:89,
AUTHOR = {K. Hornik and M. Stinchcombe and H. White},
     TITLE = {Multilayer feedforward networks are universal approximators},
   JOURNAL = {Journal Neural Networks},
    VOLUME = {2},
      YEAR = {1989},
     PAGES = {359--366}
}

@article{1layerKammonen,
title = {Adaptive random {F}ourier features with {M}etropolis sampling},
author = {A. Kammonen and J. Kiessling and P. Plecháč and M. Sandberg and A. Szepessy},
journal = {Foundations of Data Science},
volume = {2},
pages = {309-332},
year = {2020}
}

@article{kammonen2023smaller,
  title={Smaller generalization error derived for a deep residual neural network compared with shallow networks},
  author={Kammonen, A. and Kiessling, J. and Plech{\'a}{\v{c}}, P. and Sandberg, M. and Szepessy, A. and Tempone, R.},
  journal={IMA Journal of Numerical Analysis},
  volume={43},
  pages={2585--2632},
  year={2023},
  publisher={Oxford University Press}
}

@article{Davis_etal:2025,
author = {O. Davis and G. Geraci and M. Motamed},
title = {Deep Learning without Global Optimization by Random {F}ourier Neural Networks},
journal={SIAM J. Scientific Computing},
volume={47},
pages={C265--C290},
year={2025}
}

@article{DavisMotamed:2024,
  author    = {O. Davis and M. Motamed},
  title     = {Approximation Power of Deep Neural Networks: An explanatory mathematical survey},
  journal   = {arXiv preprint arXiv:2207.09511},
  year      = {2024}
}

@article{Delvos1993,
  author    = {F.-J. Delvos},
  title     = {Hermite interpolation with trigonometric polynomials},
  journal   = {BIT Numerical Mathematics},
  volume    = {33},
  number    = {1},
  pages     = {113--123},
  year      = {1993},
  doi       = {10.1007/BF01990347},
  publisher = {Springer},
  url       = {https://doi.org/10.1007/BF01990347}
}

@article{GottliebShu1997,
  author    = {D. Gottlieb and C.-W. Shu},
  title     = {On the {G}ibbs' Phenomenon and Its Resolution},
  journal   = {SIAM Review},
  volume    = {39},
  pages     = {644--668},
  year      = {1997}
}

@article{Boyd2005trouble,
  author    = {J. P. Boyd},
  title     = {Trouble with {G}egenbauer Reconstruction for Defeating {G}ibbs' Phenomenon: {R}unge Phenomenon in the Diagonal Limit of {G}egenbauer Polynomial Approximations},
  journal   = {Journal of Computational Physics},
  volume    = {204},
  year      = {2005},
  pages     = {253--264}
}

@article{AdcockHansenShadrin2014,
  author    = {B. Adcock and A. C. Hansen and A. Shadrin},
  title     = {A Stability Barrier for Reconstructions from Fourier Samples},
  journal   = {SIAM Journal on Numerical Analysis},
  volume    = {52},
  pages     = {1252--1293},
  year      = {2014}
}

@article{Hewitt1979,
  author    = {E. Hewitt and R. E. Hewitt},
  title     = {The {G}ibbs-{W}ilbraham Phenomenon: An Episode in {F}ourier Analysis},
  journal   = {Historia Mathematica},
  volume    = {21},
  year      = {1979},
  pages     = {129--160}
}

@article{Tadmor2007,
  author    = {E. Tadmor},
  title     = {Filters, Mollifiers and the Computation of the {G}ibbs' Phenomenon},
  journal   = {Acta Numerica},
  volume    = {16},
  year      = {2007},
  pages     = {305--378}
}

@article{GelbTanner2006,
  author    = {A. Gelb and J. Tanner},
  title     = {Robust Reprojection Methods for the Resolution of the {G}ibbs Phenomenon},
  journal   = {Applied and Computational Harmonic Analysis},
  volume    = {20},
  year      = {2006},
  pages     = {3--25}
}

@article{JungShizgal2004,
  author    = {J.-H. Jung and B. D. Shizgal},
  title     = {Generalization of the Inverse Polynomial Reconstruction Method in the Resolution of the {G}ibbs Phenomenon},
  journal   = {Journal of Computational and Applied Mathematics},
  volume    = {172},
  year      = {2004},
  pages     = {131--151}
}

@article{Pasquetti2004,
  author    = {R. Pasquetti},
  title     = {On Inverse Methods for the Resolution of the {G}ibbs Phenomenon},
  journal   = {Journal of Computational and Applied Mathematics},
  volume    = {170},
  year      = {2004},
  pages     = {303--315}
}

@article{HrycakGrochenig2010,
  author    = {T. Hrycak and K. Gr{\"o}chenig},
  title     = {Pseudospectral Fourier Reconstruction with the Modified Inverse Polynomial Reconstruction Method},
  journal   = {Journal of Computational Physics},
  volume    = {229},
  year      = {2010},
  pages     = {933--946}
}

@article{AdcockHansen2012,
  author    = {B. Adcock and A. C. Hansen},
  title     = {Stable Reconstructions in {H}ilbert Spaces and the Resolution of the {G}ibbs Phenomenon},
  journal   = {Applied and Computational Harmonic Analysis},
  volume    = {32},
  year      = {2012},
  pages     = {357--388}
}

@article{DriscollFornberg2001,
  author    = {T. A. Driscoll and B. Fornberg},
  title     = {A Pad{\'e}-Based Algorithm for Overcoming the {G}ibbs Phenomenon},
  journal   = {Numerical Algorithms},
  volume    = {26},
  year      = {2001},
  pages     = {77--92}
}

@article{Beckermann_etal2011,
  author    = {B. Beckermann and V. Kalyagin and A. Matos and F. Wielonsky},
  title     = {How Well Does the {H}ermite--{P}ad{\'e} Approximation Smooth the {G}ibbs Phenomenon?},
  journal   = {Mathematics of Computation},
  volume    = {80},
  year      = {2011},
  pages     = {931--958}
}

@inproceedings{Rahaman_etal2019,
  author    = {N. Rahaman and A. Baratin and D. Arpit and F. Draxler and M. Lin and F. Hamprecht and Y. Bengio},
  title     = {On the Spectral Bias of Neural Networks},
  booktitle = {Proceedings of the 36th International Conference on Machine Learning (ICML)},
  year      = {2019}
}

@article{Xu_etal2019,
  author    = {Z.-Q. J. Xu and Y. Zhang and Y. Zhai and Z. Ma},
  title     = {Frequency Principle: Fourier Analysis Sheds Light on Deep Neural Networks},
  journal   = {Communications in Computational Physics},
  volume    = {28},
  number    = {5},
  year      = {2020},
  pages     = {1746--1767}
}

@article{Basri_etal2020,
  author    = {R. Basri and D. Jacobs and I. Landa and Y. Kasten},
  title     = {Frequency bias in neural networks for input of non-uniform density},
  journal   = {arXiv preprint arXiv:2002.11610},
  year      = {2020}
}

@article{Wang_etal2021,
  author    = {S. Wang and H. Zhang and L. Franceschi and J. Fu and C.-J. Hsieh},
  title     = {On the Convergence of {F}ourier Neural Operators: From Single-scale to Multiscale},
  journal   = {arXiv preprint arXiv:2106.02582},
  year      = {2021}
}

@article{Bubeck_etal2021,
  author    = {S. Bubeck and M. Sellke},
  title     = {A universal law of robustness via isoperimetry},
  journal   = {arXiv preprint arXiv:2106.04132},
  year      = {2021}
}

@article{AdcockDexter2021,
  author    = {B. Adcock and N. Dexter},
  title     = {The Gap between Theory and Practice in Function Approximation with Deep Neural Networks},
  journal   = {SIAM Journal on Mathematics of Data Science},
  volume    = {3},
  pages     = {624--655},
  year      = {2021}
}

@article{Barron1993,
  author    = {A. R. Barron},
  title     = {Universal approximation bounds for superpositions of a sigmoidal function},
  journal   = {IEEE Transactions on Information Theory},
  year      = {1993},
  volume    = {39},
  pages     = {930--945}
}

@article{LiaoMing:2025,
  author    = {Liao, Y. and Ming, P.},
  title     = {Spectral {B}arron space for deep neural network approximation},
  journal   = {arXiv preprint arXiv:2309.00788},
  year      = {2025}
}

@article{Carleson1966,
  author    = {L. Carleson},
  title     = {On convergence and growth of partial sums of {F}ourier series},
  journal   = {Acta Mathematica},
  volume    = {116},
  pages     = {135--157},
  year      = {1966}
}

@incollection{Hunt1968,
  author    = {R. A. Hunt},
  title     = {On the convergence of {F}ourier series},
  booktitle = {Orthogonal Expansions and their Continuous Analogues},
  publisher = {Southern Illinois University Press},
  year      = {1968},
  pages     = {235--255},
  note      = {Proc. Conf., Edwardsville, Ill., 1967}
}

@book{Stein1970,
  author    = {E. M. Stein},
  title     = {Singular Integrals and Differentiability Properties of Functions},
  publisher = {Princeton University Press},
  year      = {1970}
}

@book{SteinWeiss1971,
  author    = {E. M. Stein and G. Weiss},
  title     = {Introduction to {F}ourier Analysis on {E}uclidean Spaces},
  publisher = {Princeton University Press},
  year      = {1971}
}

@book{Grafakos2014,
  author    = {L. Grafakos},
  title     = {Classical {F}ourier Analysis},
  publisher = {Springer},
  year      = {2014},
  edition   = {3rd},
  series    = {Graduate Texts in Mathematics},
  volume    = {249}
}

@book{Boyd2000,
  author    = {J. P. Boyd},
  title     = {Chebyshev and {F}ourier Spectral Methods},
  edition   = {2nd},
  publisher = {Dover Publications},
  year      = {2000}
}

@article{petersen2018optimal,
  title={Optimal approximation of piecewise smooth functions using deep {ReLU} neural networks},
  author={P. Petersen and F. Voigtlaender},
  journal={Neural Networks},
  volume={108},
  pages={296--330},
  year={2018},
  publisher={Elsevier}
}

@article{yarotsky2019phase,
  title={The phase diagram of approximation rates for deep neural networks. arXiv e-prints, page},
  author={D. Yarotsky and A. Zhevnerchuk},
  journal={arXiv preprint arXiv:1906.09477},
  year={2019}
}

@article{cybenko1989approximation,
  title={Approximation by superpositions of a sigmoidal function},
  author={G. Cybenko},
  journal={Mathematics of control, signals and systems},
  volume={2},
  number={4},
  pages={303--314},
  year={1989},
  publisher={Springer}
}

@article{leshno1993multilayer,
  title={Multilayer feedforward networks with a nonpolynomial activation function can approximate any function},
  author={M. Leshno and V. Y. Lin and A. Pinkus and S. Schocken},
  journal={Neural networks},
  volume={6},
  number={6},
  pages={861--867},
  year={1993},
  publisher={Elsevier}
}
\bibliographystyle{plain}

\end{document}